\documentclass[11pt,letterpaper]{article}
\usepackage{amsfonts}
\usepackage{amsmath,amssymb,color}
\usepackage{graphicx}
\usepackage[singlespacing]{setspace}

\newcommand{\x}{\mbox{$\mathbf x$}}

\newcommand{\y}{\mbox{$\mathbf y$}}

\newtheorem{theorem}{Theorem}

\newtheorem{corollary}{Corollary}

\newtheorem{definition}{Definition}
\newtheorem{example}{Example}

\newtheorem{lemma}{Lemma}

\newtheorem{proposition}{Proposition}

\newenvironment{proof}[1][Proof]{\noindent\textbf{#1.} }{\ \rule{0.5em}{0.5em}}

\usepackage{subfigure}
\usepackage{tikz}
\newcommand{\Prob}{{\mathbb{P}}}
\newcommand{\E}{{\mathbb{E}}}

\begin{document}
\title{A Multilevel Approach towards Unbiased Sampling of Random Elliptic Partial Differential Equations}
\author{Xiaoou Li and Jingchen Liu \footnote{This research is supported in part by NSF SES-1323977 and Army Research Office W911NF-12-R-0012.}}

\maketitle
\baselineskip 18pt
\begin{abstract}
Partial differential equation is a powerful tool to characterize various physics systems. In practice, measurement errors are often present and probability models are employed to account for such uncertainties. In this paper, we present a Monte Carlo scheme that yields unbiased estimators for expectations of random elliptic partial differential equations. This algorithm combines multilevel Monte Carlo \cite{giles2008multilevel} and a randomization scheme proposed by \cite{rhee2012new,rheeunbiased}. Furthermore, to obtain an estimator with both finite variance and finite expected computational cost, we employ higher order approximations.
\end{abstract}

\section{Introduction}

Elliptic partial differential equation is a classic equation that are employed to describe various static physics systems. In practical life, such systems are usually not described precisely. 
For instance, imprecision could be due to microscopic heterogeneity or measurement errors of parameters. 
To account for this, we introduce uncertainty to the system by letting certain coefficients contain randomness. To be precise, let $U\subset R^d$ be a simply connected domain. We consider the following differential equation concerning $u: U\to R$
\begin{equation}\label{eq}
-\nabla\cdot (a(x)\nabla u(x))=f(x) \mbox{ for } x\in U,
\end{equation}
where  $f(x)$ is a real-valued function and $a(x)$ is a strictly positive function. 
Just to clarify the notation, $\nabla u(x)$ is the gradient of $u(x)$ and ``$\nabla\cdot$" is the divergence of a vector field.
For each $a$ and $f$, one solves $u$ subject to certain boundary conditions that are necessary for the uniqueness of the solution. 
This will be discussed in the sequel. 
The randomness is introduced to the system through $a(x)$ and $f(x)$. Thus, the solution $u$ as an implicit functional of $a$ and $f$ is a real-valued stochastic process living on $U$.
Throughout this paper, we consider  $d\leq 3$ that is sufficient for most physics applications.

Of interest is the distributional characteristics of $\{u(x):x\in U\}$. The solution is typically not in an analytic form of $a$ and $f$ and thus closed form characterizations are often infeasible.
In this paper, we study the distribution of $u$ via Monte Carlo. Let $C(U)$ be the set of continuous functions on $U$.
For a real-valued functional 
$$\mathcal Q : C(U) \to R$$ satisfying certain regularity conditions, we are interested in computing
$$w_{\mathcal Q} = \E\{\mathcal Q(u)\}.$$
Such problems appear  often in  the studies of physics systems; see, for instance, \cite{de2005dealing,delhomme1979spatial}.

The contribution of this paper is the development of an unbiased Monte Carlo estimator of $w_{\mathcal Q}$ with finite variance. Furthermore, the expected computational cost of generating one such estimator is finite.  
The analysis strategy is a combination of multilevel Monte Carlo and a randomization scheme. 
Multilevel Monte Carlo is a recent advance in simulation and approximation of continuous processes \cite{giles2008multilevel,cliffe2011multilevel,graham2011quasi}.
The randomization scheme is developed by \cite{rhee2012new,rheeunbiased}.
Under the current setting, a direct application of these two methods leads to either an estimator with infinite variance or infinite expected computational cost. 
This is mostly due to the fact that the accuracy of regular numerical methods of the  partial differential equations is insufficient. More precisely, the mean squared error of a discretized Monte Carlo estimator is  proportional to the square of mesh size \cite{charrier2013finite,teckentrup2013further}.
The technical contribution of this paper is to employ quadratic approximation to solve PDE under certain smoothness conditions of $a(x)$ and $f(x)$ and to perform careful analysis of the numerical solver for equation \eqref{eq}.

\paragraph{Physics applications.}
Equation \eqref{eq} has been widely used in many disciplines to describe time-independent physical problems.  The well-known Poisson equation or Laplace equation is a special case when $a(x)$ is a constant. 
In different disciplines, the solution $u(x)$ and the coefficients $a(x)$ and $f(x)$ have their specific physics meanings. 
When the elliptic PDE is used to describe the steady-state distribution of heat (as temperature), $u(x)$ carries the meaning of temperature at $x$ and the coefficient $a(x)$ is the heat conductivity.  
In the study of electrostatics, $u$ is the potential (or voltage) induced by electronic charges, $\nabla u$ is the electric field, and $a(x)$ is the permittivity (or resistance) of the medium.  
In groundwater hydraulics, the meaning of $u(x)$ is the hydraulic head (water level elevation) and $a(x)$ is the hydraulic conductivity (or permeability).  The physics laws for the above three different problems to derive the same type of elliptic PDE are called Fourier's law, Gauss's law, and Darcy's law, respectively.  
In classical continuum mechanics, equation \eqref{eq} is known as the generalized Hook's law where $u$ describes the material deformation under the external force $f$. The coefficient $a(x)$ is known as the elasticity tensor.

	In this paper, we consider that both $a(x)$ and $f(x)$ possibly contain randomness. We elaborate its physics interpretation in the context of material deformation application. 
	In the model of classical continuum mechanics, the domain $U$ is a smooth manifold denoting the physical location of the piece of material.
The displacement $u(x)$ depends on the external force $f(x)$, boundary conditions, and the elasticity tensor $\{a(x): x\in U\}$.
The elasticity coefficient $a(x)$ is modeled  as a spatially varying random field to characterize the inherent heterogeneity and uncertainties in the  physical properties of the material (such as the modulus of elasticity, c.f.~\cite{sobczyk2001stochastic,ostoja2007microstructural}).
	For example, metals, which lend themselves most readily to the analysis by means of the classical elasticity theory, are actually polycrystals, i.e., aggregates of an immense number of anisotropic crystals randomly oriented in space. Soils, rocks, concretes, and ceramics provide further examples of materials with very complicated structures.
	Thus, incorporating  randomness in $a(x)$ is necessary to take into account of the heterogeneities and the uncertainties under many situations.
	Furthermore, there may also be uncertainty contained in the external force $f(x)$.

The rest of the paper is organized as follows. In Section~\ref{SecMLMC}, we present the problem settings and some preliminary materials for the main results. Section~\ref{SecMain} presents the construction of the unbiased Monte Carlo estimator for $w_{\mathcal Q}$  and rigorous complexity analysis. Numerical implementations are included in Section~\ref{Secsim}. Technical proofs are included in the appendix.

\section{Preliminary analysis}\label{SecMLMC}

Throughout this paper, we consider equation \eqref{eq} living on a bounded domain $U\subset R^d$ with twice differentiable boundary denoted by $\partial U$.
To ensure the uniqueness of the solution, we consider the Dirichlet boundary condition 
\begin{equation}\label{dirichlet}
u(x)=0,\quad \mbox{for $x\in$ $\partial U$.}
\end{equation}
%To account for uncertainty in the system or imperfect measurement, we introduce randomness to the system. 
We let both exogenous functions $f(x)$ and $a(x)$ be random processes, that is, 
$$f(x,\omega) : U \times \Omega \to R \quad \mbox{and} \quad a(x,\omega) : U \times \Omega \to R$$ 
where $(\Omega, \mathcal{F}, \Prob)$ is a probability space. 
To simplify notation, we omit the second argument and  write $a(x)$ and $f(x)$.
As an implicit  function of the input processes $a(x)$ and $f(x)$,  the solution $u(x)$ is also a stochastic process  living on $U$.
We are interested in computing the distribution of $u(x)$ via Monte Carlo.
In particular, for some functional 
$$\mathcal Q : C(\bar U) \to R$$
satisfying certain regularity conditions that will be specified in the sequel, we compute the expectation
\begin{equation}
w_{\mathcal Q} = \E[\mathcal Q(u)]
\end{equation}
by Monte Carlo. The notation $\bar U$ is the closure of domain $U$ and $C(\bar U)$ is the set of real-valued continuous functions on $\bar U$.

%In what follows, we  combine  multilevel Monte Carlo and a randomization scheme to compute $w_\mathcal Q$.
Let $\hat Z$ be an estimator (possibly biased) of $\E\mathcal Q(u)$. The mean square error (MSE)
\begin{equation}\label{MSE}
\E(\hat Z-w_{\mathcal Q})^2= Var(\hat Z) +\{\E(\hat Z)-w_{\mathcal Q}\}^2.
\end{equation}
consists of a bias term and a variance term.
For the Monte Carlo estimator in this paper, the bias  is removed via a randomization scheme combined with multilevel Monte Carlo. To start with, we present the basics  of multilevel Monte Carlo and the randomization scheme.

\subsection{Multilevel Monte Carlo}\label{secMLMC}

Consider a biased estimator of $w_{\mathcal Q}$ denote by $Z_n$. 
In the current context, $Z_n$ is the estimator corresponding to some numerical solution  based on certain discretization scheme, for instance, $Z_n = \mathcal Q(u_n)$ where $u_n$ is the solution of the finite element method.
The subscript $n$ is a generic index of the discretization size. The detailed construction of $Z_n$ will be provided in the sequel.
As $n \to \infty$, the estimator becomes unbiased, that is,
$$\E(Z_n ) \to w_{\mathcal Q}.$$
Multilevel Monte Carlo is based on the following telescope sum 
\begin{equation}\label{sum}
w_{\mathcal Q} = \E(Z_0) + \sum_{i=0}^\infty \E(Z_{i+1} -Z_i).
\end{equation}
One may choose $Z_0$ to be some simple constant. Without loss of generality, we choose  $Z_0\equiv 0$ and thus the first term vanishes. The advantage of writing $w_{\mathcal Q}$ as the telescope sum is that one is often able to construct $Z_{i}$ and $Z_{i+1}$ carefully such that they are appropriately coupled and the variance of $Y_i = Z_{i+1} -Z_i$ decreases fast as $i$ tends infinity.
Let 
\begin{equation}\label{delta}
\Delta_i = \E(Z_{i+1} -Z_i)
\end{equation}
be estimated by 
$$\hat \Delta_i = \frac{1}{n_i} \sum_{j=1}^{n_i} Y_i^{(j)}$$
where $Y_i^{(j)}$, $j=1,...,n_i$ are independent replicates of $Y_i$.
The multilevel Monte Carlo estimator is 
\begin{equation}\label{mlmc}
\hat Z = \sum_{i=1}^I \hat \Delta _i
\end{equation}
where $I$ is a large integer truncating the infinite sum  \eqref{sum}.

\subsection{An unbiased estimator via a randomization scheme}\label{secunbiasedef}

In the construction of the multilevel Monte Carlo estimator \eqref{mlmc}, the truncation level $I$ is always finite and therefore the estimator is always biased.
	In what follows, we present an estimator with the bias removed. It is constructed based on the telescope sum of the multilevel Monte Carlo estimator and a randomization scheme that is originally proposed by \cite{rhee2012new,rheeunbiased}.

	Let $N$ be a positive-integer-valued random variable that is independent of $\{Z_i\}_{i=1,2,\dots}$. Let  $p_n=\Prob(N=n)$ be the probability mass function of $N$ such that $p_n>0$ for all $n>0$. The following identity holds trivially
 \begin{equation*}
 w_{\mathcal Q}=\sum_{i=1}^{\infty}\E(Z_{n}-Z_{n-1})=\sum_{n=1}^{\infty} \frac{\E[Z_{n}-Z_{n-1};N=n]}{p_n}=\E\Big(\frac{Z_{N}-Z_{N-1}}{p_N}\Big).
 \end{equation*}
 Therefore, an unbiased estimator of $w_{\mathcal Q}$ is given by
  \begin{equation}\label{unbiasedest}
  \tilde{Z}=\frac{Z_{N}-Z_{N-1}}{p_N}.
  \end{equation}
 Let $\tilde{Z}_{i},i=1,...,M$ be independent copies of $\tilde{Z}$. The averaged estimator
   \begin{equation*}
   \tilde {Z}_M=\frac{1}{M}\sum_{i=1}^M \tilde{Z}_i
   \end{equation*}
   is unbiased for $w_{\mathcal Q}$ with variance $Var(\tilde{Z})/M$ if finite. 
   
%   The unbiased estimator $\tilde{Z}$ is a generalization of the multilevel Monte Carlo method in the sense that if $$N_l=\frac{p_l}{\sum_{l=1}^L p_l}\times(\sum_{l=1}^L N_l), l=1,...,L,$$
%   then the multilevel Monte Carlo method is  approximating the conditional expectation $E(\hat{Z}^{U}|\max_{i=1,...,M}N_i\leq L).$

We provide a complexity analysis of the estimator $\tilde Z$. This consists of the  calculation of the variance of  $\tilde Z$ and of the computational cost to generate $\tilde Z$.
We start with the second moment 
\begin{equation}\label{secondmomentsum}
\E(\tilde Z^2 ) = \E\Big[  \frac{(Z_N - Z_{N-1})^2}{p_N^2}\Big ] = \sum _{n=1}^\infty \frac{\E(Z_n - Z_{n-1})^2}{p_n}.
\end{equation}
In order to have finite second moment, it is almost necessary to choose the random variable $N$ such that 
\begin{equation}\label{secondmoment}
p_n > n \E(Z_n - Z_{n-1})^2 \mbox{ for all $n$ sufficiently large.} 
\end{equation}
Furthermore, $p_n$ must also satisfy the natural constraint that 
$$\sum_{n=1}^\infty p_n =1,$$
which suggests $p_n<n^{-1}$ for sufficiently large $n$.
Combining with \eqref{secondmoment}, we have
\begin{equation}\label{constraint}
n^{-1}>p_n>n \E(Z_n-Z_{n-1})^2
\end{equation}
Notice that we have not yet specified a discretization method, thus \eqref{constraint} can typically be met by appropriately indexing the mesh size. For instance, in the context of solving PDE numerically, one may choose the mesh size converging to $0$ at a super exponential rate with $n$ (such as $e^{-n^2}$) and thus $\E(Z_n - Z_{n-1})^2$ decreases sufficiently fast that allows quite some flexibility in choosing $p_n$. Thus, constraint \eqref{constraint} alone can always be satisfied and it is not intrinsic to the problem. It is the combination with the following constraint that forms the key issue.

We now compute the expected computational cost for generating $\tilde Z$. Let $c_n$ be the computational cost for generating $Z_n - Z_{n-1}$. Then, the expected cost is
\begin{eqnarray}\label{expectedcost}
C = \sum_{i=1}^n p_n c_n.
\end{eqnarray}
In order to have $C$ finite, it is almost necessary that 
\begin{equation}\label{cost}
p_n < n^{-1} c_n^{-1}.
\end{equation}
Based on the above calculation, if the estimator $\tilde Z$ has a finite variance and a finite expected computation time, then $p_n$ must satisfy both \eqref{constraint} and \eqref{cost}, which suggests
\begin{equation}\label{intrinsic}
 \E(Z_n - Z_{n-1})^2< n^{-2} c_n^{-1}.
\end{equation}
That is, one must be able to construct a coupling between $Z_n$ and $Z_{n-1}$ such that \eqref{intrinsic} is in place.
In Section \ref{SecMain}, we provide detailed complexity analysis for the random elliptic PDE illustrating the challenges and presenting the solution.

\subsection{Function spaces and norms}\label{secnotation}
In this section, we present a list of notation that will be frequently used in later discussion.
Let $U\subset R^d$ be a bounded open set. We define the following spaces of functions.
\begin{eqnarray*}
C^k(\bar U)&=&\{u: \bar U\to R|u \mbox{ is $k$-time continuously differentiable} \}\\
L^p(U)&=&\{u:U\to R| \int_U |u(x)|^p dx<\infty\}\\
L^p_{loc}(U)&=&\{u:U\to R| u\in L^p(K) \mbox{ for any compact subset $K\subset U$}\}\\
C^{\infty}_c(U)&=&\{u:U\to R| u \mbox{ is infinitely differentiable with a compact support that is a subset of $U$}\}.
\end{eqnarray*}
\begin{definition}
For $u,w\in L^1_{loc}(U)$ and a multiple index $\alpha$,  we say $w$ is the $\alpha$-weak derivative of $u$, and write
$$
D^{\alpha}u=w
$$ if
$$
\int_U u D^{\alpha}\phi dx=(-1)^{|\alpha|}\int_U w\phi dx \mbox{ for all } \phi\in C^{\infty}_c(U),
$$
where $D^{\alpha}\phi$ in the above expression denote the usual $\alpha$-partial derivative of $\phi$. 
\end{definition}
If $u\in C^{k}(\bar U)$ and  $|\alpha|\leq k$, then the $\alpha$-weak derivative and the usual partial derivative are the same. Therefore, we can write $D^{\alpha}\phi$ for both continuously differentiable and weakly differentiable functions without ambiguous.

We further define norms $\|\cdot\|_{C^k(\bar U)}$ and $\|\cdot\|_{L^p(U)} $ on $C^k(\bar U)$ and $L^p(U)$ respectively as follows.
\begin{equation}\label{defcknorm}
\|u\|_{C^k{(\bar U})}= \sup_{|\alpha|\leq k, x\in \bar{U}} |D^{\alpha} u(x)|,
\end{equation}
and
\begin{equation}\label{deflpnorm}
\|u\|_{L^p(U)}=\Big(\int_U |u|^p dx\Big)^{1/p}.
\end{equation}
We proceed to the definition of Sobolev space $H^k(U)$ and $H^k_{loc}(U)$
\begin{equation}\label{defhk}
H^k(U)=\{u:U\to R| D^{\alpha}u\in L^2(U) \mbox{ for all multiple index $\alpha$ such that } |\alpha|\leq k \},
\end{equation}
and 
$$
H^{k}_{loc}(U)=\{u: U\to R|~ u|_V\in H^{k}(V) \mbox{ for all } V\subsetneq U  \}
$$
For $u\in H^k(U)$, the norm $\|u\|_{H^k(U)}$ and semi-norm $|u|_{H^k(U)}$ are defined as 
\begin{equation}\label{defhknorm}
\|u\|_{H^k(U)}=\Big(\sum_{|\alpha|\leq k}\|D^{\alpha}u\|^2_{L^2(U)}\Big)^{1/2}
,
\end{equation}
and
\begin{equation}\label{defhkseminorm}
|u|_{H^k(U)}=\Big(\sum_{|\alpha|= k}\|D^{\alpha}u\|^2_{L^2(U)}\Big)^{1/2}
.
\end{equation}
We define the space $H^1_0(U)$ as
\begin{equation}\label{defh01}
H^1_0(U)=\{u\in H^1(U): u(x)=0 \mbox{ for } x \in \partial U\}.
\end{equation}
On the space $H^1_0(U)$ the norm $\|\cdot\|_{H^1(U)}$ and the semi-norm $|\cdot|_{H^1(U)}$ are equivalent.

\subsection{Finite element method for partial differential equation}

We briefly describe the finite element method for partial differential equations.
The weak solution $u\in H^1_0(U)$ to \eqref{eq} under the Dirichlet boundary condition \eqref{dirichlet} is defined through the following variational form
\begin{equation}\label{weakPDE}
b(u,v)=L(v)\mbox{ for all } v\in H^1_0(U),
\end{equation}
where we define the  bilinear and linear forms 
$$b(u,v)=\int_U a(x)\nabla u(x)\cdot\nabla v(x)dx\quad  \mbox{ and } \quad L(v)=\int_U f(x)v(x)dx,$$
and ``$\cdot$" is the vector inner product. When the coefficients $a$ and $f$ are sufficiently smooth, say, infinitely differentiable, the weak solution $u$ becomes a strong solution.
The key step of the finite element method is to approximate the infinite dimensional space $H^1_0(U)$ by some finite dimensional linear space $V_n=\mathrm{span}\{\phi_1,...,\phi_{L_n}\}$, where  $L_n$ is the dimension of $V_n$. The approximate solution $u_n\in V_n$ is defined through the set of equations 
\begin{equation}\label{femequation}
b(u_n,v)=L(v) \mbox{ for all } v\in V_n.
\end{equation}
Both sides of the above equations are linear in $v$. Then, \eqref{femequation} is equivalent to
\begin{equation*}
b(u_n,\phi_i)=L(\phi_i) \mbox{ for } i=1,...,L_n.
\end{equation*}
We further write $u_n=\sum_{i=1}^{L_n} d_i \phi_i$ as a linear combination of the basis functions. Then, \eqref{femequation} is equivalent to solving linear equations
\begin{equation}\label{linearequation}
\sum_{j=1}^{L_n} d_j b(\phi_j,\phi_i)=L(\phi_i) \mbox{ for } i=1,...,L_n.
\end{equation}
The basis functions $\phi_1,...,\phi_{L_n}$ are often chosen such that \eqref{linearequation} is a sparse linear system. Solving a sparse linear system requires a computational cost of order $O(L_n\log(L_n))$ as $L_n\to\infty.$
%Second, the space $H_0^1(D)$ is replaced by certain 
%With the integration equation \eqref{weakPDE}, the problem of seeking the solution to \eqref{eq} is transformed to finding the solution $u\in H^1_0(D)$ that satisfies \eqref{weakPDE}. 
%In order to perform numerical approximation to the solution of \eqref{weakPDE}, the space $H^1_0(D)$ is replaced by a 
%finite-dimensional subspace $V_h$. The numerical solution $u_h$ is obtained by solving the following finite dimensional equation, where $h$ is a generic notation indicating the mesh size of the discretization.
%\begin{equation}\label{eqfeml}
%\mbox{ Find } u_h\in V_h \mbox{ such that }
%b(u_h,v_h)=L(v_h),\mbox{ for all } v_h \in V_h.
%\end{equation}

\section{Main results}\label{SecMain}
	In this section, we present the  construction of  $\tilde Z$ and its complexity analysis. We use finite element method to solve the PDE numerically and then  construct $Z_n$. 
	To illustrate the challenge, we start with the complexity analysis of $\tilde Z$ based on usual finite element method with linear basis functions, with which we show that \eqref{constraint} and \eqref{cost} cannot be satisfied simultaneously.
	Thus, $\tilde Z$ either has infinite variance or has infinite expected computational cost.
	We improve upon this by means of quadratic approximation under smoothness assumptions on  $a$ and $f$. The estimator $\tilde Z$  thus can be  generated in constant time and has a finite variance.

\subsection{Error analysis of finite element method}\label{seclinear}
%To start with, we describe the general procedure of finite element method to find the solution $u\in V$ of the variational equation
%$$
%b(u,v)=L(u,v) \mbox{ for all } v\in V,
%$$
%where $v$ is called test function, and $V$ is a space of test functions. In the case of 

%The heart of finite element method is the approximation of the  $V$ by some finite linear space $V_n$

\paragraph{Piecewise linear basis functions.}

A popular choice of  $V_n$ is the space of piecewise linear functions defined on a triangularization $\mathcal{T}_{n}$ of $U$. 
In particular, $\mathcal{T}_{n}$ is a partition of $U$ that is each element of $\mathcal{T}_n$ is a triangle partitioning $U$. The maximum edge length of triangles is proportional to $2^{-n}$ and  $V_n$ is the space of all the piecewise linear functions over $\mathcal{T}_n$ that vanish on the boundary $\partial U$. The dimension  of $\mathcal{T}_n$ is $L_n=O(2^{dn})$.
Detailed construction of $\mathcal T _n$ and piecewise linear basis functions is provided in Appendix \ref{sectri} and Example \ref{extr} therein.

	Once a set of basis functions has been chosen, the coefficients $d_i$'s are solved according to the linear equations \eqref{linearequation} and the numerical solution is given by 
	$$u_n(x)=\sum_{i=1}^{L_n} d_i \phi_i(x).$$
	For each functional $\mathcal Q$, the biased estimator is 
	\begin{equation*}
	Z_n = \mathcal Q (u_n).
	\end{equation*}
	It is important to notice that, for different $n$, $u_n$ are computed based on the same realizations of $a$ and $f$. Thus, $Z_n $ and $Z_{n-1}$ are coupled.
	
	We now proceed to verifying \eqref{intrinsic} for linear basis functions.  The dimension of $V_n$ is of order $L_n=O(2^{dn})$ where $d=dim(U)$.
	We consider the case when $\mathcal{Q}$ is a functional that involves weak derivatives of $u$. For instance, $\mathcal{Q}$ could be in the form  $q(|\cdot|_{H^1(U)})$ for some smooth function $q$  and $Z=\mathcal{Q}(u)$, where $|\cdot|_{H^1(U)}$ is defined as in \eqref{defhkseminorm}.

 According to Proposition 4.2 of \cite{charrier2013finite}, 
 under the conditions that $\E[\frac{1}{\min_{x\in U} a^p(x)}]<\infty$, $\E(\|a\|^p_{C^1(\bar U)})<\infty$, and $\E(\|f\|^p_{L^2(U)})<\infty$ for all $p>0$,  $\E(Z_n-Z_{n-1})^2=O(2^{-2n})$ if $u_{n}$ and $u_{{n-1}}$ are computed using the same sample of $a$ and $f$. The condition \eqref{intrinsic} becomes
$$n 2^{-2(n-1)}<n^{-1}2^{-dn}|\log 2^{-nd}|^{-1}.$$
A simple calculation yields that the above inequality holds only if $d=1$.
Therefore, \emph{it is impossible to pick  $p_n$ such that the estimator $\tilde Z$  has a finite variance and a finite expected computational cost using the finite element method with linear basis functions if $d\geq 2$. }
The one-dimensional case is not of great interest given that $u$ can be solved explicitly.
To establish \eqref{intrinsic} for higher dimensions, we need a faster convergence rate of the PDE numerical solver.

\paragraph{Quadratic basis functions.}

We improve accuracy of the finite element method by means of piecewise polynomial basis functions under smoothness conditions on $a(x)$ and $f(x)$. Classical results (e.g. \cite{knabner2003numerical}) show that finite element method with polynomial basis functions provides more accurate results than that with piecewise linear basis functions. 
We obtain similar results for random coefficients. 
Define the minimum and maximum of  $a(x)$ as
\begin{equation*}
a_{\min}=\min_{x\in\bar{U}} a(x)\mbox{ and }  a_{\max}=\max_{x\in\bar{U}} a(x).
\end{equation*}
We make the following assumptions on the random coefficients $a(x)$ and $f(x)$.
\begin{itemize}
\item [A1.] $a_{\min}> 0$ almost surely and $\E(1/a^p_{\min})<\infty$, for all $p\in (0,\infty)$.
\item[A2.] $a$ is almost surely continuously twice differentiable  and $\E(\|a\|^p_{C^2(\bar U)})<\infty$ for all $p \in (0,\infty)$.
\item[A3.] $f\in H^1(U)$ almost surely and $\E(\|f\|^p_{H^1(U)}) <\infty$ for all $p\in (0,\infty)$.
\item[A4.] There exist non-negative constants $p'$ and $\kappa_q$ such that for all $w_1,w_2\in H^{1}_0(U)$,  $$|\mathcal{Q}(w_1)-\mathcal{Q}(w_2)|\leq \kappa_q\max\{\|w_1\|^{p'}_{H^1(U)},\|w_2\|^{p'}_{H^1}\}\|w_1-w_2\|_{H^1(U)}.$$
\end{itemize}
With the assumptions A1-A4, we are able to construct an unbiased estimator for $w_{\mathcal Q}=\E[\mathcal Q(u)]$ with both finite variance and finite expected computational time.

%We adopt the notations used in Chapter 3.4 of \cite{knabner2003numerical}. 

Let $k$ be a positive interger and $\mathcal{T}_n$ be a regular triangularization of the domain $U$ with mesh size $\sup_{K\in\mathcal{T}_n} diam(K)=O(2^{-n})$, whose detailed definition is provided in Appendix~\ref{sectri} and let $V^{(k)}_n$ be the set of piecewise continuous polynomials on $\mathcal{T}_n$ that have degrees no more than $k$ and vanish on the boundary of $U$.
To be more specific, $V^{(k)}_n$ is defined as follows
\begin{eqnarray*}
V^{(k)}_n=\Big\{v\in C(\bar{U}): v|_{K} \mbox{ is a polynomial with degree no more than $k$,}\\  \mbox{ for each } K\in \mathcal{T}_n  \mbox{ and }v|_{\bar{U}\setminus{D_n}}=0\Big\},
\end{eqnarray*}
where $D_n=int(\cup_{K\in \mathcal{T}_n,K\subset \bar U}K)$ and $int(A)$ denotes the interior of the set $A.$
An approximate solution $u_n^{(k)}$ is obtained by solving \eqref{femequation} with $V_n=V^{(k)}_n$, that is,
\begin{equation}\label{eqfem}
 u^{(k)}_n\in V^{(k)}_h \mbox{ such that }
b(u^{(k)}_n,v)=L(v),\mbox{ for all } v \in V^{(k)}_n.
\end{equation}
In what follows, we present a bound of the convergence rate of $\|u^{(k)}_n-u\|_{H^1(U)}$, where $u$ is the solution to \eqref{weakPDE} and $u_n^{(k)}$ is the solution to \eqref{eqfem}. 
We start with the existence and the uniqueness of the solution.
 Notice that $a(x)$  is bounded below by positive random variables $a_{\min}$ and above by $a_{\max}$. According to Lax-Milgram Lemma, \eqref{weakPDE} has a unique solution almost surely.
\begin{lemma}[ \cite{charrier2013finite}, Lemma 2.1.]\label{h1solution}
Under assumptions A1-A3,  \eqref{weakPDE} has a unique  solution  $u\in H_0^1(U)$ almost surely and 
\begin{equation*} %\label{H1norm}
\|u\|_{H^1(U)}\leq \kappa \frac{\|f\|_{L^2(U)}}{a_{\min}}.
\end{equation*}

\end{lemma}
The next theorem establishes the convergence rate of the approximate solution $u^{(k)}_n$ to the exact solution $u$.

\begin{theorem}\label{thmfek}
Let $u^{(k)}_n$ be the solution to \eqref{eqfem}. For $\mathrm{dim}(U)\leq 3$ with a $(k+1)$-time differentiable boundary $\partial U$, if $a(x)\in C^{k}(\bar U)$ and $f(x)\in H^{k-1}(U)$ for some positive integer $k$, then we have 
\begin{equation}\label{H1accuracy}
\|u-u_n^{(k)}\|_{H^1(U)}=O\Big( \kappa(a,k)\|f\|_{H^{k-1}(U)}2^{-kn}\Big),
\end{equation}
where the constant $\kappa(a,k)$ is defined as $$\kappa(a,k)=\frac{\max(\|a\|_{C^k(\bar U)},1)^{\frac{k^2}{2}+\frac{9}{2}k-\frac{1}{2}}} {\min(a_{\min},1)^{\frac{k^2}{2}+\frac{7}{2}k+\frac{3}{2}}}.$$
\end{theorem}
The proof of Theorem~\ref{thmfek} is given in Appendix~\ref{secproof}.
In our analysis, we focus on the case $k=2$ that is sufficient for our analysis. We state the results for this special case.

\begin{corollary}
For $dim(U)\leq 3$, if $a(x)\in C^2(\bar U)$ and $f(x)\in H^1(U)$, then
\begin{equation*}
\|u-u_n^{(2)}\|_{H^1(U)} =O\Big( \frac{\max(\|a\|_{C^2(\bar U)},1)^{10.5}}{\min(a_{\min},1)^{10.5}}\|f\|_{H^{1}(U)} 2^{-2n}\Big).
\end{equation*}

%\begin{enumerate}
%\item[(i)] 
%\begin{equation*}
%|u-u_h|_{H^1(U)}\lesssim \frac{\|a\|^{6.5}_{C^2(\bar U)}}{a_{\min}^{9.5}}\|f(\cdot)\|_{H^{1}(U)} h^2,
%\end{equation*}
%\item[(ii)]
%$$
%\|u-u_h\|_{L_2} \lesssim \frac{\|a\|^{13.5}_{C^2(\bar U)}}{a_{\min}^{18.5}}h^{3}.
%$$
%\end{enumerate}
\end{corollary}

\paragraph{Quadrature Error Analysis.}\label{secquadrature}

The numerical solution $u^{(k)}_n$ in \eqref{eqfem} requires the evaluation of the integrals  $b(w,v)=\sum_{K\in \mathcal{T}_n }\int_{K} a(x)\nabla w(x) \cdot\nabla v(x) dx$ and $L(v)=\sum_{K\in \mathcal{T}_n}\int_{K}f(x) v(x) dx$. This requires  generating the entire continuous random fields $a(x)$ and $f(x)$. For the evaluation of these integrals we apply quadrature approximation.

In our analysis, we use linear approximation to $a(\cdot)$ and $f(\cdot)$ on each simplex $K\in\mathcal{T}_n,$ then the integrals can be calculated analytically. We will give a careful analysis for the quadrature error of $b(w,v)$. The analysis for $L(v)$ is similar and thus is omitted. 

Let $\tilde{a}(\cdot)$ be the linear interpolation of $a(\cdot)$ given its values on vertices such that for all simplex $K\in\mathcal{T}_n$,
$
\tilde{a}(x)=a(x)$ if $x$ is a vertice of $K$, and  $\tilde{a}|_{K}$ is linear.
Such interpolation is easy to obtain using piecewise linear basis functions discussed in Section \ref{seclinear}. We define the bilinear form induced by $\tilde a(\cdot)$ as 
$$
\tilde{b}_n(w,v)=\sum_{K\in \mathcal{T}_n} \int_{K}\tilde{a}( x) \nabla w(x)\cdot\nabla v(x) dx,
$$
and denote by $\tilde{u}_n\in V^{(2)}_n$ the solution to 
\begin{equation}\label{eqfemqua}
\tilde{b}_n(\tilde{u}_n,v)=L(v), \mbox{ for all $v\in V^{(2)}_n.$}
\end{equation}
The next theorem establishes the convergence rate for $\tilde u_n$ to the solution $u$. The proof for Theorem~\ref{propquadr} is given in Appendix~\ref{secproof}. 
%Applying first Strang Lemma (Lemma 3.38 of \cite{knabner2003numerical}), we have
\begin{theorem}\label{propquadr}
For $dim(U)\leq 3$, if $a(x)\in C^2(\bar U)$ and $f(x)\in H^1(U)$, then  
\begin{equation*} %\label{quaerror}
\|u-\tilde{u}_{n}\|_{H^1(U)}=O\Big(\frac{\min(\|a\|_{C^2(\bar{U})},1)^{11.5}}{\min(a_{\min},1)^{11.5}}\|f\|_{H^1(U)}2^{-2n}\Big).
\end{equation*}

\end{theorem}
This accuracy is sufficient for the unbiased estimator to have finite variance and finite expected stopping time.
Similarly, we let $\tilde f$ be the linear interpolation of $f$ on $\mathcal T_n$ and define $\tilde L(v) = \sum_{K \in \mathcal T_n }\int_K \tilde f(x) v(x) dx.$ We redefine $\tilde{u}_n$ such that 
\begin{equation}\label{eq1}
\tilde{b}_n(\tilde{u}_n,v)=\tilde L(v), \mbox{ for all $v\in V^{(2)}_n.$}
\end{equation}
Similar approximation results as that of Theorem \ref{propquadr} can be obtained. We omit the repetitive details.

%
%\subsection{Error of Generating the Random Coefficients}
%In practice, approximation of $a$  is used when the computational cost of simulating it exactly is high. We consider the error induced by this approximation.
%Assume $\hat{a}$ is used to approximate the random field $a$, and $\hat{u}_h$ is obtained by solving
%$$
%\hat{b}(\hat{u}_h(\cdot),v_h)=L(v_h), \mbox{ for all $v_h\in V_h,$}
%$$
%where $\hat{b}(w_h,v_h)=\sum_{K\in \mathcal{T}_h} \int_{K}\hat{a}( x) \nabla w_h\cdot\nabla v_h dx$.
%According to the first Strang lemma \eqref{strang}, we have
%\begin{equation*}
%|u(\cdot)-\tilde{u}_{h}(\cdot)|_{H^1(U)}\lesssim \frac{\|a(\cdot)\|^{3.5}_{C^2(\bar{U})}}{a_{\min}^{9.5}}h^2+\sup_{x\in U}|a(x)-\hat{a}(x)|.
%\end{equation*}
%For the case where $\log a(x)$ is a Gaussian random field with the covariance function \eqref{gaussiancov}, a popular choice of $\hat{a}$ is the truncated KL-expansion. \cite{charrier2013finite} 

\subsection{Construction of the unbiased estimator}\label{secunbiased}
In this section, we apply the results obtained in Section~\ref{secquadrature} to construct an unbiased estimator with both finite variance and finite expected computational cost through \eqref{unbiasedest}. 
We start with providing an upper bound of $\E[\mathcal{Q}(u)-\mathcal{Q}(\tilde u_{n})]^2$.
\begin{proposition}\label{propbeta}
Under assumptions A1-A4, we have
\begin{equation}\label{outputerror}
\E[\mathcal{Q}(u)-\mathcal{Q}(\tilde u_{n})]^2=O(\kappa_q 2^{-4n}),
\end{equation}
where $u$ is the solution to \eqref{weakPDE} and $\tilde{u}_n$ is the solution to \eqref{eq1}, and $\kappa_q$ the Lipschitz constant appeared in condition A4.
\end{proposition}
\begin{proof}
The proof is a direct application of  Lemma~\ref{h1solution}, Theorem~\ref{propquadr} and A4 and therefore is omitted.
\end{proof}

We proceed to the construction of the unbiased estimator $\tilde{Z}$ via  \eqref{unbiasedest}. Choose
$$\Prob(N=n)=p_n\propto 2^{-\frac{4+d}{2}n}.$$
For each $n$, let $\tilde u_{n-1}$ and $\tilde u_{n}$ be defined as in \eqref{eq1} with respect to the same $a$ and $f$.
Notice that the computation of $\tilde u_n$ requires the values of $a$ and $f$ only on the vertices of $\mathcal{T}_n$.
Then, $Z_{n-1}$ and $Z_n$ are given by $Z_{n-1}=\mathcal Q(\tilde u_{n-1})$ and $Z_n=\mathcal Q(\tilde u_{n})$. 
With this coupling, according to Proposition \ref{propbeta}, we have that 
$$
\E(Z_n-Z_{n-1})^2\leq 2\E[\mathcal Q(\tilde u_{n})-\mathcal Q(u)]^2+2\E[\mathcal Q(\tilde u_{{n-1}})-\mathcal Q(u)]^2= O( 2^{-4n}).
$$
According to equation \eqref{secondmomentsum}, for $d=\mathrm{dim}(U)\leq 3$, we have
\begin{equation*}
\E(\tilde Z^2)\leq \sum_{n=1}^\infty 2^{-4n}/2^{-(4+d)n/2}<\infty.
\end{equation*}
Furthermore,  \eqref{eq1} requires solving $O(2^{dn})$ sparse linear equations. The computational cost of obtaining $u_{n}$ is $O(n2^{dn})$. According to \eqref{expectedcost}, the expected cost of generating a single copy of $\tilde Z$ is
$$
\E(C)=\sum_{n=1}^\infty p_n c_n\leq \sum_{i=1}^{\infty} n2^{dn}\cdot 2^{-(4+d)n/2}<\infty.
$$ 
This guarantees that the unbiased estimator $\tilde Z$ has a finite variance and can be generated in finite expected time.
%
%Now we analyze the expected computational cost of the unbiased estimator defined in \eqref{secunbiased}. Let $u_{h}$ be the finite element approximation with mesh size $h$ of triangularization, and $Z_n$, and $\tilde{Z}$ be defined in Section \ref{secunbiased}. According to Proposition \ref{outputerror}, we have
%\begin{equation*}
%E(Z_n-Z)^2\lesssim 2^{-4n}.
%\end{equation*}
%The computational cost for obtaining $Z_n$ is the computational cost of solving the PDE using finite element method, which is of order $O(h^d\log(h))=O(n2^{dn})$.
%Therefore, if we pick $p_n$ to be proportional to $2^{-3n}$, then the expected computational cost of generating a single unbiased estimator is finite.

\section{Simulation Study}\label{Secsim}

\subsection{An illustrating example}\label{secex1}
We start with a simple example for which closed form solution is available and therefore we are able to check the accuracy of the simulation. Let $U=(0,1)^2$, $f(x)=\sin(\pi x_1)\sin(\pi x_2)$ and $a(x)=e^{W}$, where $W$ is a standard normal distributed random variable. In this example, the exact solution to \eqref{eq} is 
\begin{equation}
u(x_1,x_2)=(2\pi^2)^{-1} e^{-W}\sin(\pi x_1)\sin(\pi x_2).
\end{equation}
We are interested in the output functional $\mathcal{Q}(u)=|u|^2_{H^1(U)}$ whose expectation is in a closed form.
$$\E|u|^2_{H^1(U)}=\E[(8\pi^2)^{-1}e^{-2W}]=(8\pi^2)^{-1}e^2\approx 0.0936.$$
Let  $p_n=0.875\times 0.125^{n}$ and $Z_n=\mathcal{Q}(\tilde u_{n})$ for $n> 0$.
Here $Z_0$ is not a constant and we estimate $\E(Z_0)$ and $\E(Z-Z_0)$ separately. 
To be more precise, we first estimate $\E(Z_0)$ using the usual Monte Carlo estimate with $10000$ replicates and obtain $\hat{Z}_0=0.036$ with standard error $0.0024$. The estimator according to \eqref{unbiasedest} is
\begin{equation}\label{z0not0}
\tilde{Z}=\hat{Z}_0+\frac{Z_{N}-Z_{N-1}}{p_{N}}.
\end{equation}
We perform Monte Carlo simulation with $M=10000$ replications. The averaged estimator is $0.0939$ with the standard deviation $0.0036$. Figure~\ref{fig:ex1hist}  shows the histogram of samples of $\tilde{Z}$ and $\log\tilde{Z}$.

In order to conform our analytical results, we simulate the expectation for $\E(Z_{n}-Z)^2$ and $c_n$ for $n=0,..,5,$ using $1000$ Monte Carlo sample for each of them. The scatter plot of $n$ and $\log_2(\E(Z_{n}-Z)^2)$  is shown in Figure~\ref{fig:ex1mse}. The slope of the regression line in this graph is $-3.85$, which is close to the theoretical value $-4.$ The scatter plot of $n$ and $\log_2 c_n$ is shown in Figure~\ref{fig:ex1timecost}. The slope of the regression line in this graph is $2.031$, which is close to the theoretical value $2.$
% % % for different h use a table to show E(|h|-||),
% % % Var(h-h-1),h, computational cost and p_h
% % % also show a graph
\begin{figure}[!ht]
\centering
\subfigure{}\includegraphics[scale = 0.3]{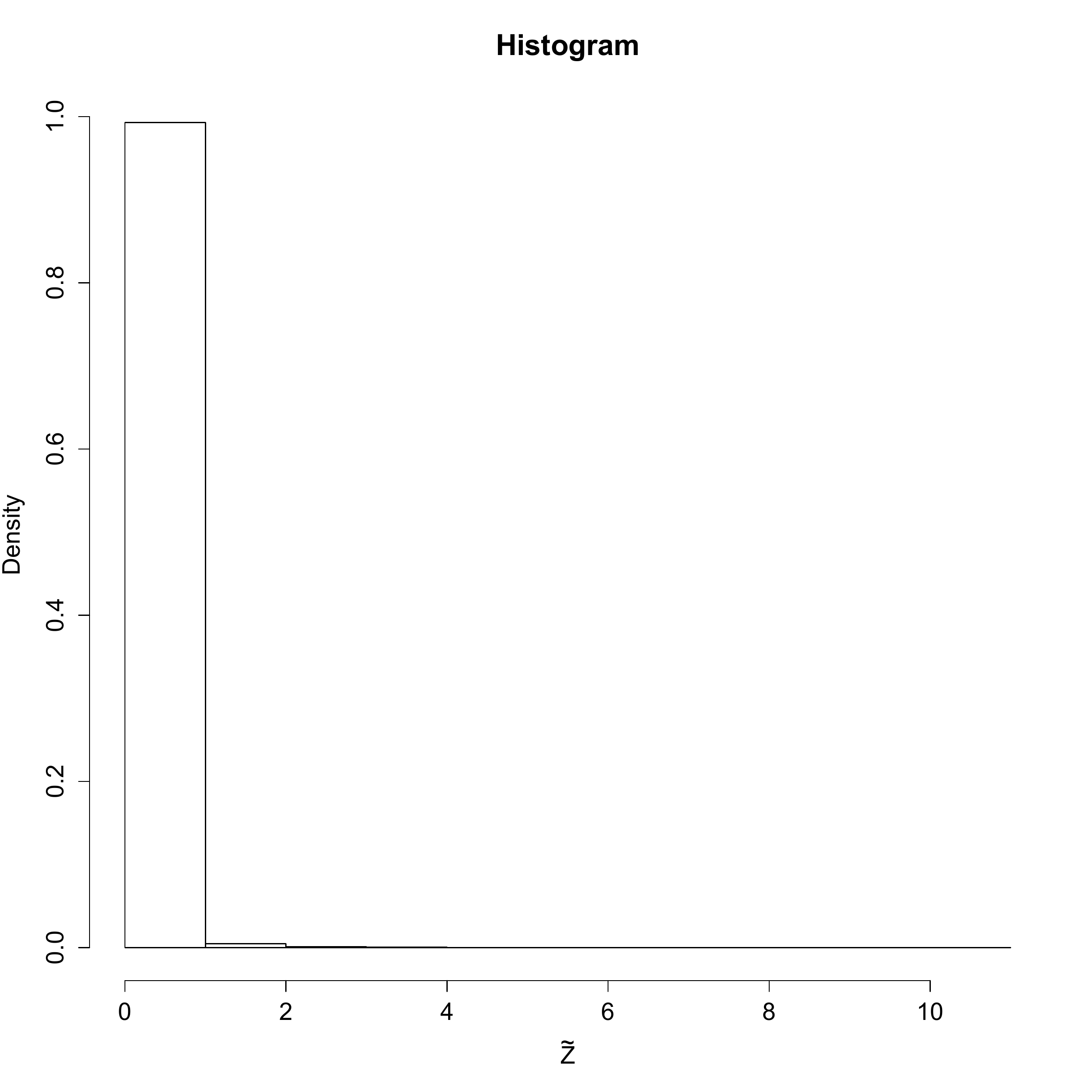}
\subfigure{\label{fig:ex1loghist}}\includegraphics[scale = 0.3]{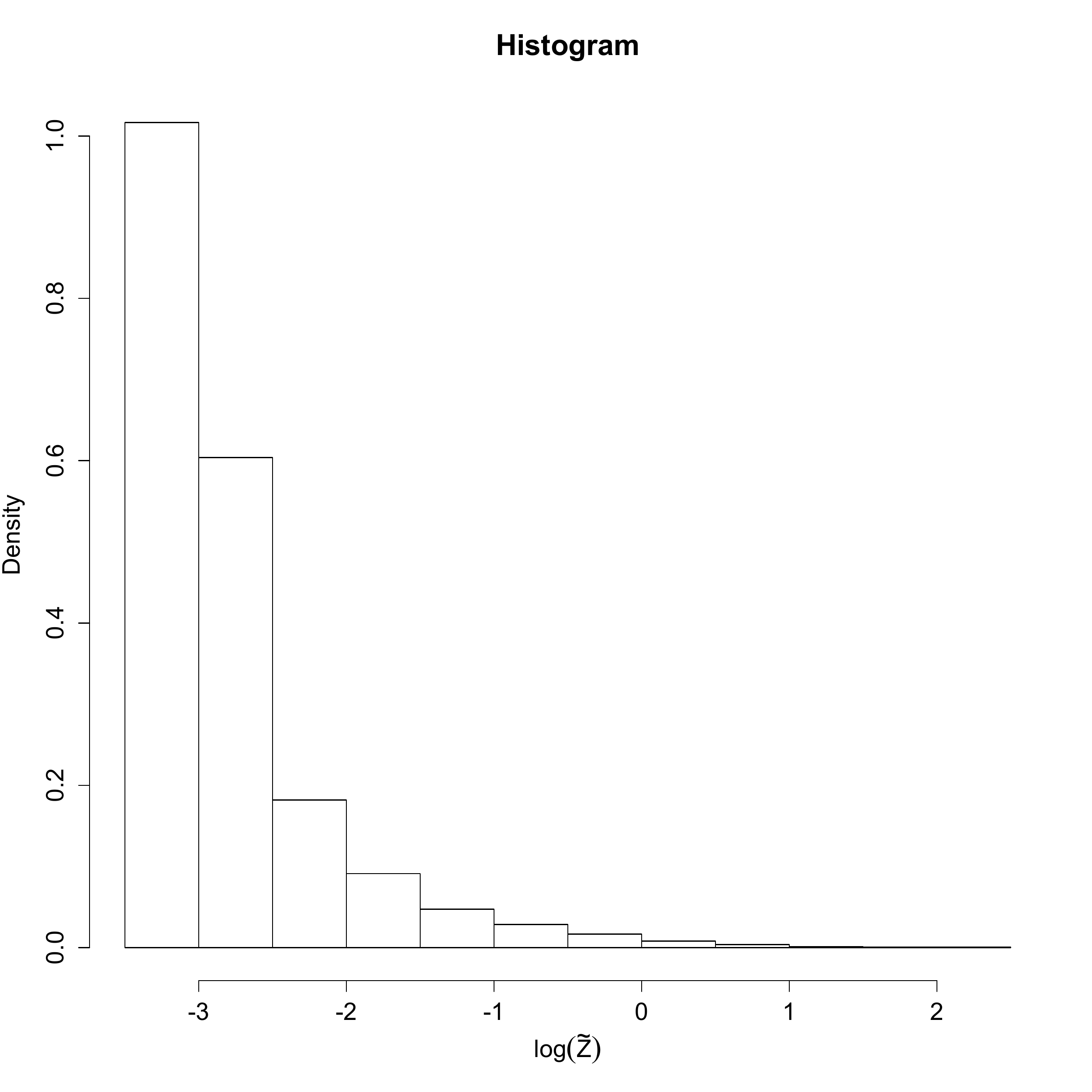}
\caption{Histogram of Monte Carlo sample of $\tilde{Z}$ and $\log\tilde{Z}$ that are defined in Section \ref{secex1}. }
\label{fig:ex1hist}
\end{figure}

\begin{figure}[!ht]
\centering
\includegraphics[scale = 0.4]{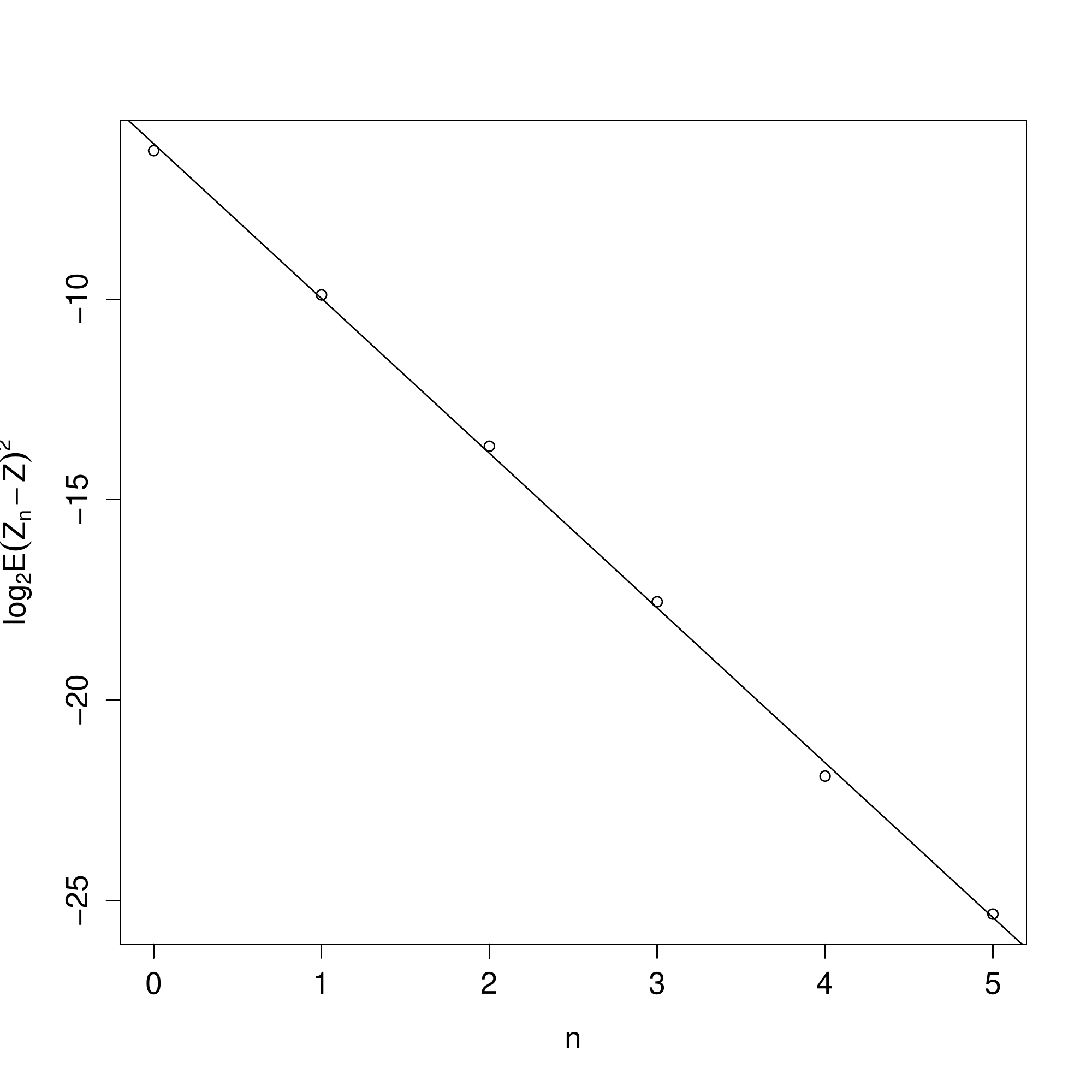}
\caption{Scatter plots for $n$ against $\log(\E(Z-Z_{n})^2)$  in the example in Section \ref{secex1}.}
\label{fig:ex1mse}
\end{figure}

\begin{figure}[!ht]
\centering
\includegraphics[scale = 0.4]{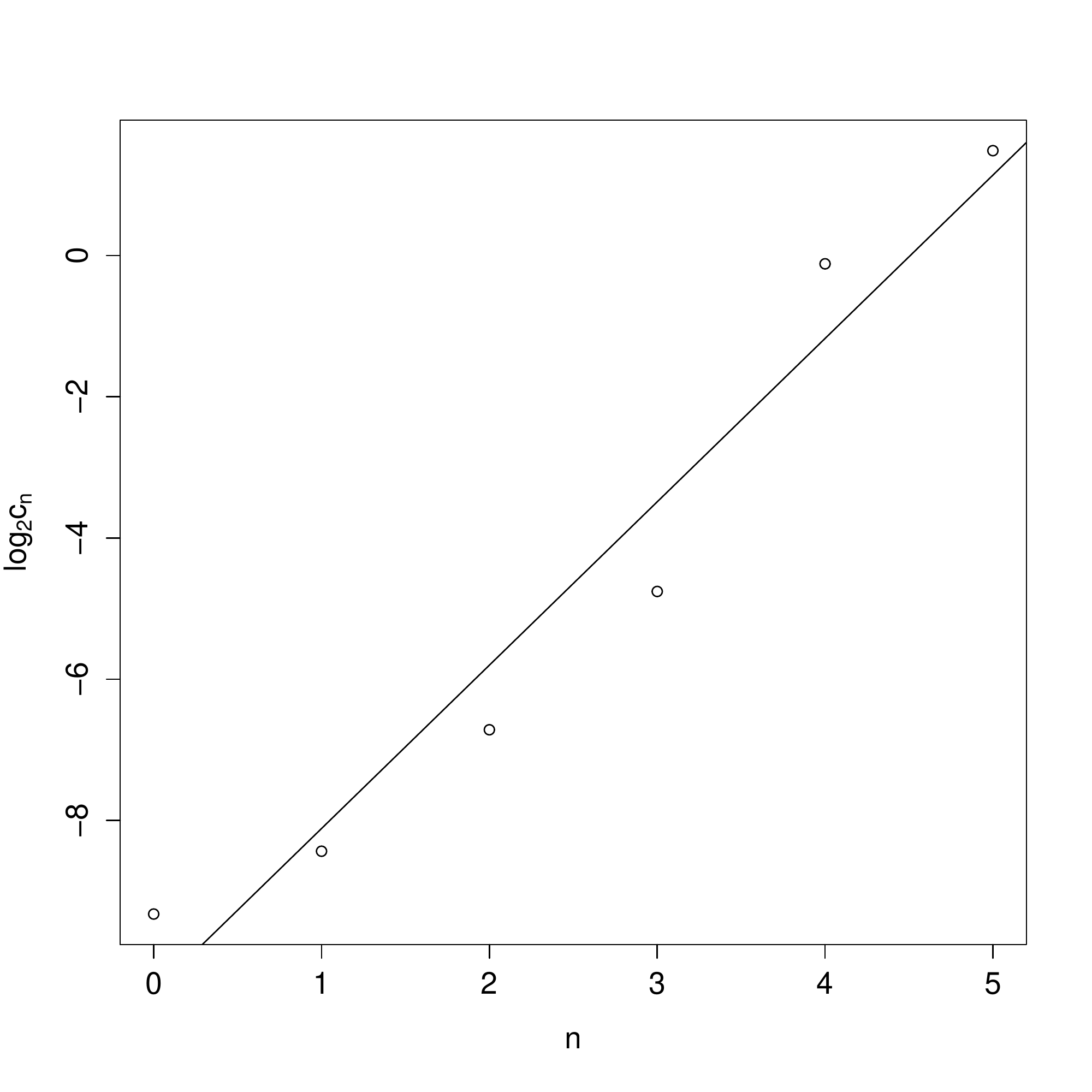}
\caption{Scatter plots for  $n$ against $\log(c_n)$ in the example in Section \ref{secex1}.}
\label{fig:ex1timecost}
\end{figure}
\subsection{Log-normal random field with Gaussian covariance kernel}
Here we let $U=(0,1)^2$, $f=1$, and $\log a$ be  modeled as a Gaussian random field with the covariance function
$$
Cov(\log(a(x)),\log(a(y)))=\exp(-{|x-y|^2}/{\lambda}).
$$ 
with $\lambda=0.03$.
Such a log-normal random field is infinitely differentiable and satisfies assumptions A1 and A2.
 We use the circulant embedding method (see \cite{dietrich1997fast}) to generate the random field $\log a$ exactly. We use the same estimator as in \eqref{z0not0} and consider  $\mathcal{Q}(u)=|u|^2_{H^1(U)}$.
We perform Monte Carlo simulation for $M=100000$ replications. The averaged estimator  for the expectation $\E\mathcal{Q}(u)$ is $0.0428$ and the standard deviation is $0.0032$ for the averaged estimator. Figure ~\ref{fig:ex2hist} shows the histogram of the Monte Carlo sample.
\begin{figure}[!ht]
\centering
\includegraphics[scale = 0.4]{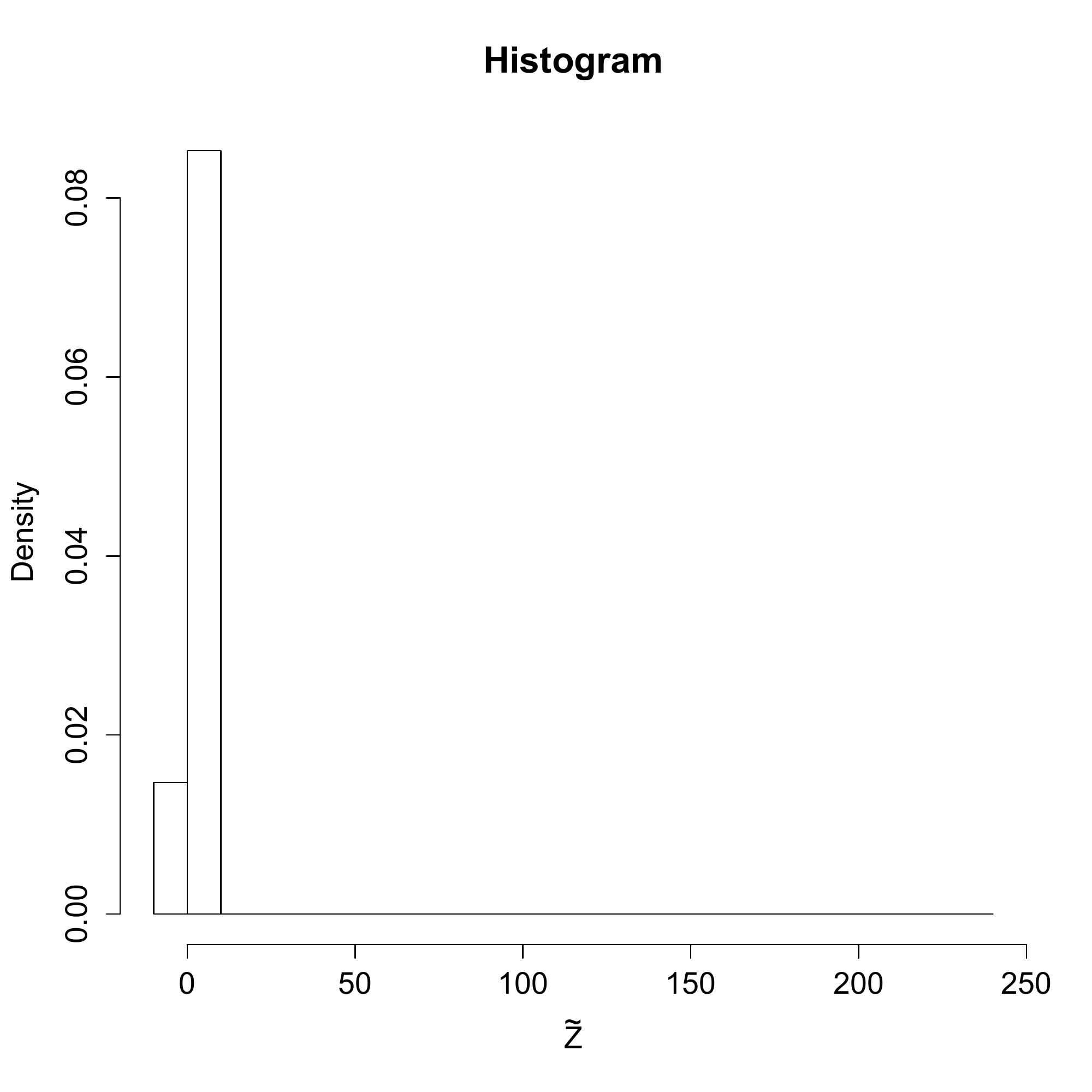}
\caption{Histogram of Monte Carlo sample of $\tilde{Z}$ when $\log a$ has a Gaussian covariance.}
\label{fig:ex2hist}
\end{figure}

\bibliographystyle{apalike}
\bibliography{pde}

\begin{thebibliography}{}

\bibitem[Charrier et~al., 2013]{charrier2013finite}
Charrier, J., Scheichl, R., and Teckentrup, A.~L. (2013).
\newblock Finite element error analysis of elliptic pdes with random
  coefficients and its application to multilevel monte carlo methods.
\newblock {\em SIAM Journal on Numerical Analysis}, 51(1):322--352.

\bibitem[Cliffe et~al., 2011]{cliffe2011multilevel}
Cliffe, K., Giles, M., Scheichl, R., and Teckentrup, A.~L. (2011).
\newblock Multilevel monte carlo methods and applications to elliptic pdes with
  random coefficients.
\newblock {\em Computing and Visualization in Science}, 14(1):3--15.

\bibitem[De~Marsily et~al., 2005]{de2005dealing}
De~Marsily, G., Delay, F., Gon{\c{c}}alv{\`e}s, J., Renard, P., Teles, V., and
  Violette, S. (2005).
\newblock Dealing with spatial heterogeneity.
\newblock {\em Hydrogeology Journal}, 13(1):161--183.

\bibitem[Delhomme, 1979]{delhomme1979spatial}
Delhomme, J. (1979).
\newblock Spatial variability and uncertainty in groundwater flow parameters: A
  geostatistical approach.
\newblock {\em Water Resources Research}, 15(2):269--280.

\bibitem[Dietrich and Newsam, 1997]{dietrich1997fast}
Dietrich, C. and Newsam, G.~N. (1997).
\newblock Fast and exact simulation of stationary gaussian processes through
  circulant embedding of the covariance matrix.
\newblock {\em SIAM Journal on Scientific Computing}, 18(4):1088--1107.

\bibitem[Evans, 1998]{evans1998partial}
Evans, L.~C. (1998).
\newblock {\em Partial differential equations}.
\newblock Providence, Rhode Land: American Mathematical Society.

\bibitem[Giles, 2008]{giles2008multilevel}
Giles, M.~B. (2008).
\newblock Multilevel monte carlo path simulation.
\newblock {\em Operations Research}, 56(3):607--617.

\bibitem[Graham et~al., 2011]{graham2011quasi}
Graham, I.~G., Kuo, F.~Y., Nuyens, D., Scheichl, R., and Sloan, I.~H. (2011).
\newblock Quasi-monte carlo methods for elliptic pdes with random coefficients
  and applications.
\newblock {\em Journal of Computational Physics}, 230(10):3668--3694.

\bibitem[Knabner and Angermann, 2003]{knabner2003numerical}
Knabner, P. and Angermann, L. (2003).
\newblock Numerical methods for elliptic and parabolic partial differential
  equations.

\bibitem[Ostoja-Starzewski, 2007]{ostoja2007microstructural}
Ostoja-Starzewski, M. (2007).
\newblock {\em Microstructural randomness and scaling in mechanics of
  materials}.
\newblock CRC Press.

\bibitem[Rhee and Glynn, 2012]{rhee2012new}
Rhee, C.-h. and Glynn, P.~W. (2012).
\newblock A new approach to unbiased estimation for sde's.
\newblock In {\em Proceedings of the Winter Simulation Conference}, page~17.
  Winter Simulation Conference.

\bibitem[Rhee and Glynn, 2013]{rheeunbiased}
Rhee, C.-h. and Glynn, P.~W. (2013).
\newblock Unbiased estimation with square root convergence for sde models.

\bibitem[Sobczyk and Kirkner, 2001]{sobczyk2001stochastic}
Sobczyk, K. and Kirkner, D.~J. (2001).
\newblock {\em Stochastic modeling of microstructures}.
\newblock BIRKH\"AUSER.

\bibitem[Teckentrup et~al., 2013]{teckentrup2013further}
Teckentrup, A., Scheichl, R., Giles, M., and Ullmann, E. (2013).
\newblock Further analysis of multilevel monte carlo methods for elliptic pdes
  with random coefficients.
\newblock {\em Numerische Mathematik}, 125(3):569--600.

\end{thebibliography}

\newpage

\appendix

\section{Proof of the Theorems}\label{secproof}
In this section, we provide technical proofs of Theorem \ref{thmfek} and Theorem \ref{propquadr}. Throughout the proof we will use $\kappa$ as a generic notation to denote large and not-so-important constants whose value may vary from place to place. Similarly, we use $\varepsilon$ as a generic notation for small positive constants.  
\begin{proof}[Proof of Theorem~\ref{thmfek}]
Using C\'ea's lemma (Theorem 2.17 of \cite{knabner2003numerical}), the convergence rate of finite element method can be bounded according to the regularity property of $u$.
\begin{equation}\label{cea}
\|u-u_n^{(k)}\|_{H^1(U)}\leq (\frac{a_{\max}}{a_{\min}})^{1/2}\inf_{v\in V_n^{(k)}}\|u-v\|_{H^1(U)}.
\end{equation}
Furthermore, if  $u\in H^{k+1}(U)$, standard interpolation result (See Theorem 3.29 of\\ \cite{knabner2003numerical}) gives an upper bound of the right-hand side of the inequality \eqref{cea}
\begin{equation}\label{interpolation}
\inf_{v\in V_n^{(k)}}\|u-v\|_{H^1(U)}=O \Big(2^{-kn}\|u\|_{H^{k+1}(U)}\Big).
\end{equation}
According to \eqref{cea} and \eqref{interpolation}, it is sufficient to derive an upper bound of $\|u\|_{H^{k+1}(U)}$, which is given in the following proposition.
\begin{proposition}\label{higherorderregularity}
	Under the setting of Theorem~\ref{thmfek}, we have

\begin{equation*}
\|u\|_{H^{k+1}(U)}\leq  \kappa\frac{\max(\|a\|_{C^k(\bar U)},1)^{\frac{k^2}{2}+\frac{9}{2}k-1}} {\min(a_{\min},1)^{\frac{k^2}{2}+\frac{7}{2}k}}\Big(\|f\|_{H^{k-1}(U)}+\|u\|_{L^2(U)}\Big).
\end{equation*}
\end{proposition}
Combining  \eqref{interpolation} and Proposition~\ref{higherorderregularity} we have
\begin{equation}\label{bound1}
\inf_{v\in V_n^{(k)}}\|u-v\|_{H^1(U)}\leq 2^{-kn} \kappa\frac{\max(\|a\|_{C^k(\bar U)},1)^{\frac{k^2}{2}+\frac{9}{2}k-\frac{1}{2}}} {\min(a_{\min},1)^{\frac{k^2}{2}+\frac{7}{2}k+\frac{1}{2}}}\Big(\|f\|_{H^{k-1}(U)}+\|u\|_{L^2(U)}\Big).
\end{equation}
According to the Poincar\'e's lemma (Theorem 2.18 of \cite{knabner2003numerical})  
$$
\|u\|_{L^2(U)}\leq \kappa \|u\|_{H^1(U)}.
$$
Thanks to Lemma~\ref{h1solution}, the above display can be further bounded by $$\|u\|_{L^2(U)}\leq \kappa\frac{\|f\|_{L^2(U)}}{a_{\min}}.$$
We complete the proof by combining the above expression and \eqref{bound1}.
\end{proof}

%\begin{proof}[Proof to Theorem~\ref{thml2}]
%The usual duality argument (See Theorem 3.37 of \cite{knabner2003numerical}) shows that
%\begin{equation}\label{duality}
%\|u-u_h\|_{L_2}\lesssim \kappa_r(a,k) \|u-u_h\|_{H^1(U)}\cdot h,
%\end{equation}
%where $\kappa_r(a,k)$ is defined in Proposition \ref{higherorderregularity}. Combining \eqref{duality} and \eqref{H1norm} gives the desired result.
%\end{proof}

\bigskip

\begin{proof}[Proof of Theorem~\ref{propquadr}]
According to Lemma 3.12 of \cite{knabner2003numerical}, 
\begin{equation}\label{strang}
\|u-\tilde{u}_{n}\|_{H^1(U)}\leq \inf_{v\in V_n^{(k)}}\Big\{(1+\frac{a_{\max}}{a_{\min}})\|u-v\|_{H^1(U)}+\frac{1}{a_{\min}}\sup_{w\in V_n^{(k)}}\frac{|b(v,w)-\tilde{b}(v,w)|}{\|w\|_{H^1(U)}}\Big\}.
\end{equation}
Notice that $\tilde{a}$ is a linear interpolation of $a$ with $O(2^{-n}) $ mesh size, so the difference between $\tilde{a}$ and $a$ is $O(\|a\|_{C^2(\bar U)}2^{-2n})$ and
\begin{eqnarray*}
|b(v,w)-\tilde{b}(v,w)|
&=&|\sum_{K\in \mathcal{T}_n }\int_{K}(\tilde{a}(x)-a(x))\nabla v\cdot\nabla w dx|\\
&\leq& \kappa\|a\|_{C^2(\bar{U})}2^{-2n}\sum_{K\in \mathcal{T}_n}\int_{K}|\nabla v|\cdot |\nabla w| dx.\\
\end{eqnarray*}
Therefore, for all $v\in V_n^{(k)}$, we have
\begin{equation*}
\|u-\tilde{u}_{n}\|_{H^1(U)}\leq \kappa (1+\frac{a_{\max}}{a_{\min}})\|u-v\|_{H^1(U)}+\frac{\|a\|_{C^2(\bar{U})}}{a_{\min}} \|v\|_{H^1(U)}2^{-2n}.
\end{equation*}
Let $v=u_n^{(2)}$. According to Lemma~\ref{h1solution}, Theorem~\ref{thmfek}, and the above display, we complete the proof.
\end{proof}

\bigskip

For the rest of the section, we provide the proof for Proposition~\ref{higherorderregularity}.  Proposition~\ref{higherorderregularity} is similar to Theorem 5 in Chapter 6.3 of \cite{evans1998partial} but we provide explicitly the dependence of constants on $a$ and $f$.

\bigskip

\begin{proof}[Proof of Proposition~\ref{higherorderregularity}]
We prove Proposition~\ref{higherorderregularity} by proving the following result for the weak solution $w\in H^1_0(U)$ to a more general PDE,
\begin{equation}\label{generalPDE2}
\left\{
\begin{array}{rcl}
-\nabla\cdot(A\nabla w)&=&f \mbox{ in } U\\
w&=&0 \mbox{ on } \partial U,
\end{array}
\right.
\end{equation}
where $A(x)=(A_{ij}(x))_{1\leq i,j\leq d}$ is a symmetric positive definite matrix function in the sense that there exist
 $A_{\min}>0$ satisfying 
 \begin{equation}\label{pd}
 \xi^TA(x)\xi\geq A_{\min}|\xi|^2
 \end{equation}
  for all $x\in \bar{U}$ and $\xi\in R^d$.  Assume that $A_{ij}(x)\in C^{k}(\bar U)$ for all $i,j=1,...,d.$
Then, it is sufficient to show that
\begin{equation}\label{eqreg}
\|w\|_{H^{k+1}(U)}\leq \kappa_r(A,k)\Big(\|f\|_{H^{k-1}(U)}+\|w\|_{L^2(U)}\Big),
\end{equation}
where $\kappa_r(A,k)=\kappa\frac{\max(\|A\|_{C^k(\bar U)},1)^{\frac{k^2}{2}+\frac{9}{2}k-1}} {\min(A_{\min},1)^{\frac{k^2}{2}+\frac{7}{2}k}}$, and 
$
\|A\|_{C^{k}(\bar U)}=\max_{1\leq i, j\leq d}\|A_{ij}\|_{C^k(\bar U)}
$.

\medskip

Let  $B^0(0,r)$ denote the open ball $\{x: |x|<r\}$ and $R^d_+=\{x\in R^d:x_d> 0\}$.
We will first prove that if
$
U=B^{0}(0,r)\cap R^d_+
$ and $V=B^0(0,t)\cap R^d_+$, then for all $t$ and $r$ such that 
  and $0<t<r$, 
\begin{equation}\label{boundarylocalbound}
\|w\|_{H^{m+2}(V)}\leq \kappa_{r,t,m+1}\frac{\max(\|A\|_{C^k(\bar U)},1)^{\frac{(m+1)^2}{2}+\frac{9}{2}(m+1)-1}} {\min(A_{\min},1)^{\frac{(m+1)^2}{2}+\frac{7}{2}(m+1)}} \Big(\|f\|_{H^{m}(U)}+\|w\|_{L^2(U)}\Big),
\end{equation}
where $\kappa_{r,t,m+1}$ is a constant depending only on $r$, $t$, and $m+1$.
The following lemma establish \eqref{boundarylocalbound} for $m=0$.
\begin{lemma}[Boundary $H^2$-regularity]\label{h2boundary}
Assume $\partial U$ is twice differentiable and $A(x)$ satisfies \eqref{pd}. Assume that $A_{ij}(x)\in C^{1}(\bar U)$ for all $i,j=1,...,d.$
Suppose furthermore $w\in H_0^1(U)$ is a weak solution to the elliptic PDE with boundary condition \eqref{generalPDE2}.
%\begin{equation}\label{generalPDE2}
%\left\{
%\begin{array}{rcl}
%-\nabla\cdot(A\nabla w)&=&f \mbox{ in } U\\
%w&=&0 \mbox{ on } \partial U,
%\end{array}
%\right.
%\end{equation}
Then $w\in H^2(U)$ and
$$
\|w\|_{H^2(U)}\leq\kappa \frac{\max(\|A\|_{C^1(\bar U)},1)^4}{\min(A_{\min},1)^4}\Big( \|f\|_{L^2(U)}+\|w\|_{L^2(U)}\Big).
$$
\end{lemma}
We establish \eqref{boundarylocalbound} by induction.
Suppose for some $m$
\begin{equation}\label{eqindboundary}
\|w\|_{H^{m+1}(W)}\leq \kappa_{t,s,m}\frac{\max(\|A\|_{C^k(\bar U)},1)^{\frac{m^2}{2}+\frac{9}{2}m-1}} {\min(A_{\min},1)^{\frac{m^2}{2}+\frac{7}{2}m}}(\|f\|_{H^{m-1}(U)}+\|w\|_{L^2(U)}),
\end{equation}
where 
\begin{equation}\label{eqnewW}
W=B^0(0,s)\cap R^d_+, \mbox{ and } s=\frac{t+1}{2}.
\end{equation}
Since $w$ is a weak solution to \eqref{generalPDE2}, it satisfies the integration equation
\begin{equation}\label{weakgeneralPDE}
 \int_D \nabla w(x)^T A(x) \nabla v(x) dx=\int_D f(x) v(x) dx, \mbox{ for all } v\in H^{1}_0(U).
\end{equation}
Let $\alpha=(\alpha_1,...,\alpha_d)$ be a multiple index with such that $\alpha_d=0$ and $|\alpha|=m$. 
 We consider the multiple weak derivative ${\bar w}=D^{\alpha}w$ and investigate the PDE that ${\bar w}$ satisfies. For any ${\bar v} \in C_c^{\infty}(W)$, where $C_c^{\infty}(W)$ is the space of infinitely differentiable functions that have compact support in $W$, we plug $v=(-1)^{|\alpha|}D^{\alpha}{\bar v}$ into \eqref{weakgeneralPDE}. With some calculations, we have
\begin{equation*}
\int_W (\nabla {\bar w}(x))^T A(x)\nabla {\bar v}(x)=\int_W {\bar f}(x){\bar v}(x)dx,
\end{equation*}
where 
\begin{equation}\label{eqnewf}
{\bar f}=D^{\alpha}f-\sum_{\beta\leq \alpha,\beta\neq\alpha}{\alpha\choose\beta} \Big[-\nabla\cdot(D^{\alpha-\beta}A\nabla D^{\beta}w)\Big].
\end{equation}
Consequently, ${\bar w}$ is a weak solution to the PDE
\begin{equation}\label{eqnewPDE}
-\nabla\cdot(A\nabla {\bar w})={\bar f} \quad\mbox{   for $x$ in } W.
\end{equation}
Furthermore, we have the boundary condition ${\bar w}(x)=0$ for  $x\in \partial W\cap\{x_d=0\}$.
By the induction assumption \eqref{eqindboundary} and \eqref{eqnewf}, we have
\begin{equation}\label{eqfnewbound}
\|{\bar f}\|_{L^2(W)}\leq \|f\|_{H^{m}(U)}+\kappa_{t,s,m}\frac{\max(\|A\|_{C^k(\bar U)},1)^{\frac{m^2}{2}+\frac{9}{2}m-1}} {\min(A_{\min},1)^{\frac{m^2}{2}+\frac{7}{2}m}}\|A\|_{C^{m+1}(\bar U)}\Big(\|f\|_{H^{m-1}(U)}+\|w\|_{L^2(U)}\Big).
\end{equation}
According to the definition of ${\bar w}$, we have
\begin{equation}\label{eqnewwbound}
\|{\bar w}\|_{L^2(W)}\leq \|w\|_{H^{m}(W)}.
\end{equation}
 Applying Lemma~\ref{h2boundary} to ${\bar w}$ with \eqref{eqfnewbound} and \eqref{eqnewwbound}, we have
\begin{equation}\label{eqsmallalpha}
\|D^{\alpha}w\|_{H^2(V)}\leq \kappa_{t,s,m}\kappa \frac{\max(\|A\|_{C^1(\bar U)},1)^4}{\min(A_{\min},1)^4}\frac{\max(\|A\|_{C^k(\bar U)},1)^{\frac{m^2}{2}+\frac{9}{2}m-1}} {\min(A_{\min},1)^{\frac{m^2}{2}+\frac{7}{2}m}}\|A\|_{C^{m+1}(\bar U)}\Big(\|f\|_{H^{m}(U)}+\|w\|_{L^2(U)}\Big).
\end{equation}
Because $\alpha$ is an arbitrary multi-index such that $\alpha_d=0$, and $|\alpha|=m$, \eqref{eqsmallalpha} implies that $D^{\beta} w\in L^2(W)$ for any multiple index $\beta$ such that $|\beta|\leq m+2$ and $\beta_d=0,1,2$. We now extend this result to multiple index $\beta$ whose last component is greater than $2$.  Suppose for all $\beta$ such that $|\beta|\leq m+2$ and $\beta_d\leq j$ , we have
\begin{equation}\label{eqboundj}
\|D^{\beta}w\|_{H^2(V)}\leq \kappa_r^{(j)} \Big(\|f\|_{H^{m}(U)}+\|w\|_{L^2(U)}\Big),
\end{equation}
where $\kappa_r^{(j)}$ is a constant depending on $A$, $m$ and $j$ that we are going to determine later. We establish the  relationship 
between $\kappa^{(j)}_r$ and $\kappa^{(j+1)}_r$.
For any $\gamma$ that is a multiple index such that $|\gamma|=m+2$ and $\gamma_d=j+1$, we use \eqref{eqboundj} to develop an upper bound for $\|D^{\gamma} w\|_{H^2(V)}$. In particular, let $\beta=(\gamma_1,..,\gamma_{d-1},j-1)$. According to the remark (ii) after Theorem 1 of Chapter 6.3 in \cite{evans1998partial},
%Lemma~\ref{lemmainterior},
 we have that
\begin{equation}\label{eqnewPDEae}
-\nabla\cdot(A\nabla (D^{\beta}w))=f^{\dagger} \mbox{  in } W \mbox{ a.e}, 
\end{equation}
where
\begin{equation}\label{eqnewf2}
f^{\dagger}=D^{\beta}f-\sum_{\delta\leq \beta,\delta\neq\beta}{\beta\choose\delta} \Big[-\nabla\cdot(D^{\beta-\delta}A\nabla D^{\delta}w)\Big].
\end{equation}
Notice that
\begin{eqnarray*}
-\nabla\cdot(A\nabla (D^{\beta}w))
=-A_{dd}D^{\gamma} w+\mbox{ sum of terms involves at most $j$ times weak derivatives of $w$}\\
\mbox{ with respect to $x_d$ and at most $m+2$ times derivatives in total.}
\end{eqnarray*}
According to  \eqref{eqboundj}, \eqref{eqnewPDEae}, \eqref{eqnewf2}, and the above display, we have
\begin{equation*}
\|D^{\gamma}w\|_{L^2(U)}\leq\kappa\frac{1}{\min(A_{\min},1)}\Big\{\|A\|_{C^{m+1}(\bar U)}\kappa_r^{(j)} \Big(\|f\|_{H^{m}(U)}+\|w\|_{L^2(U)}\Big)+\|f\|_{H^m(U)} \Big\}.
\end{equation*}
Therefore, 
\begin{equation*}
\|D^{\gamma}w\|_{L^2(U)}\leq \kappa_r^{(j+1)} \Big(\|f\|_{H^{m}(U)}+\|w\|_{L^2(U)}\Big),
\end{equation*}
where \begin{equation}\label{iterative}
\kappa_r^{(j+1)}=\kappa_r^{(j)}\frac{\max(\|A\|_{C^{m+1}(\bar U)},1)}{\min(A_{\min},1)}.
\end{equation}
The above expression provides a  relationship for $\kappa_r^{(j+1)}$ and $\kappa_r^{(j)}$.
According to \eqref{eqsmallalpha}, $$\kappa_r^{(2)}=\kappa_{t,s,m}\kappa\frac{\max(\|A\|_{C^1(\bar U)},1)^4}{\min(A_{\min},1)^4}\frac{\max(\|A\|_{C^k(\bar U)},1)^{\frac{m^2}{2}+\frac{9}{2}m-1}} {\min(A_{\min},1)^{\frac{m^2}{2}+\frac{7}{2}m}}\max(\|A\|_{C^{m+1}(\bar U)},1).$$
Using \eqref{iterative} and the above initial value for the iteration, we have
$$
\kappa_r^{(m+2)}=\kappa_{t,s,m}\kappa\frac{\max(\|A\|_{C^1(\bar U)},1)^4}{\min(A_{\min},1)^4}\frac{\max(\|A\|_{C^k(\bar U)},1)^{\frac{m^2}{2}+\frac{9}{2}m-1}} {\min(A_{\min},1)^{\frac{m^2}{2}+\frac{7}{2}m}}\max(\|A\|_{C^{m+1}(\bar U)},1)\Big\{\frac{\max(\|A\|_{C^{m+1}(\bar U)},1)}{\min(A_{\min},1)}\Big\}^m.
$$
Consequently,
\begin{equation*}
\|w\|_{H^{m+2}(V)}\leq \kappa_{t,s,m}\kappa\frac{\max(\|A\|_{C^k(\bar U)},1)^{\frac{m^2}{2}+\frac{11}{2}m+4}} {\min(A_{\min},1)^{\frac{m^2}{2}+\frac{9}{2}m+4}} \Big(\|f\|_{H^{m}(U)}+\|w\|_{L^2(V)}\Big).
\end{equation*}
Using induction, we complete the proof of \eqref{eqreg} for the case where $U$ is a half ball. 

Now we extend the result to the case that $U$ has a $C^{k+1}$ boundary $\partial U$. We first prove the theorem locally for any point $x^0\in \partial U$. Because $\partial U$ is $(k+1)$-time differentiable, with possibly relabeling, the coordinates of $x$ there exist a function $\gamma: R^{d-1}\to R$ and $r>0$ such that, 
$$
B(x^0,r)\cap U = \{x\in B(x^0,r):x_d>\gamma(x_1,...,x_{d-1}) \}.
$$
Let $\Phi=(\Phi_1,...,\Phi_d)^T: R^d\to R^d$ be a function such that
$$
\Phi_i(x)=x_i \mbox{ for } i=1,...,d-1 \mbox{ and } \Phi_d(x)=x_d-\gamma(x_1,...,x_{d-1}).
$$
Let $y=\Phi(x)$ and choose $s>0$ sufficiently small such that $$U^*=B^{0}(0,s)\cap\{y_d>0\}\subset \Phi(U\cap B(x^0,r)).$$ Furthermore, we let $V^*=B^0(0,\frac{s}{2})\cap\{y_d>0\}$ and set 
$$
w^*(y)=w(x)=w(\Phi^{-1}(y)).
$$
With some calculation, we have that $w^*$ is a weak solution to the PDE
$$
-\nabla\cdot\Big(A^*(y)\nabla w^*(y)\Big)=f^*(y),
$$
where 
$
A^*(y)=J(y)A(\Phi^{-1}(y))J^{T}(y)
$
and $J(y)$ is the Jacobian matrix for $\Phi$ with $J_{ij}(y)=\frac{\partial \Phi_i(x)}{\partial{x_j}}|_{x=\Phi^{-1}(y)}$, and $f^*(y)=f(\Phi^{-1}(y))$.
In addition, $w^*\in H^1(U^*)$ and $w^*(y)=0$ for $y\in \partial U^*\cap\{y_d=0\}$. 
It is easy to check $A^*$ is symmetric and $A^*_{ij}\in C^k(\bar U)$ for all $1\leq i,j\leq d$. Furthermore,  according to the definition of $J$ and $\Phi$, all the eigenvalues of $J(y)$ are $1$ and thus $\zeta^T A^*(y)\xi\geq A_{\min}|J^T(y)\xi|^2\geq\varepsilon A_{\min}|\xi|^2$ for all $\xi\in R^d$.
By substituting $U$, $V$, $A$, $f$ with $U^*$, $V^*$, $A^*$ and $f^*$ in \eqref{boundarylocalbound} we have
$$
\|w^*\|_{H^2(V^*)}\leq \kappa_{r}(A,k)\Big(\| w^*\|_{L^2(U^*)}+\|f^*\|_{H^{k-1}(U^*)}\Big).
$$
According to the definitions of $w^*$ and $f^*$, the above display implies
$$
\|w\|_{H^2(\Phi^{-1}(V^*))}\leq  \kappa_{r}(A,k)\Big(\| w\|_{L^2(U)}+\|f\|_{H^{k-1}(U)}\Big).
$$ 
Because $U$ is bounded, $\partial U$ is compact and thus can be covered by finitely many sets $\Phi^{-1}(V^*_1),..,\Phi^{-1}(V^*_K)$ that are constructed similarly as $\Phi^{-1}(V^*)$. We finish the proof by combining the result for points around $\partial U$ and the following Lemma~\ref{lemmainteriorhigher} for interior points. 
\begin{lemma}[Higher order interior regularity]\label{lemmainteriorhigher}
Under the setting of Lemma~\ref{h2boundary}, we assume that $\partial U$ is $C^{k+1}$, $A_{ij}(x)\in C^{k}(U)$ for all $i,j=1,...,d$, and $f\in H^{k-1}(U)$, and that $w\in H^1(U)$ is one of the weak solutions to the PDE \eqref{generalPDE2} without boundary condition. Then, $w\in H_{loc}^{k+1}(U)$. For each open set $V\subsetneqq U$
\begin{equation*}
\|w\|_{H^{k+1}(V)}\leq \kappa_i(A,k)\Big(\|f\|_{H^{k-1}(U)}+\|w\|_{L^2(U)}\Big),
\end{equation*}
where $\kappa_i(A,k)=\frac{\max(\|A\|_{C^{k}(\bar U)},1)^{3k-1}}{\min(A_{\min},1)^{2k}}\kappa$, and $\kappa$ is a constant depending on $V$.
\end{lemma}
%
%we have
%\begin{equation}\label{h2toh1bound}
%\|w\|_{H^2(U)}\leq K \frac{\max(\|A\|_{C^1(\bar{U})},1)^{4}}{\min(A_{\min},1)^{4}}\Big(\| w\|_{L^2(U)}+\|f\|_{L^2(U)}\Big).
%\end{equation}
\end{proof}

\section{Proof of supporting lemmas}
In this section, we provide the proofs for lemmas that are necessary for the proof of Proposition~\ref{higherorderregularity}.
We start with a useful lemma showing $w\in H^2_{loc}(U)$ which will be used in the proof of Lemma~\ref{h2boundary}
\begin{lemma}[Interior $H^2$-regularity]\label{lemmainterior}
Under the setting of Lemma~\ref{h2boundary}, we further assume that $A_{ij}(x)\in C^{1}(\bar U)$ for all $i,j=1,...,d,$ and $f\in L^2(U)$,
and that $w\in H^1(U)$ is one of the weak solutions to the PDE \eqref{generalPDE2} without boundary condition.
Then, $w\in H_{loc}^2(U)$. For each open subset $V\subsetneqq U$, there exist $\kappa$ depending on $V$ such that
\begin{equation*}
\|w\|_{H^2(V)}\leq \kappa \frac{\max(\|A\|_{C^1(U)},1)^2}{\min(A_{\min},1)^2}\Big(\|f\|_{L^2(U)}+\|w\|_{L^2(U)}\Big),
\end{equation*}
where we define the norm $\|A\|_{C^1(\bar{U})}=\max_{1\leq i,j\leq d}\|A_{ij}\|_{C^1(\bar U)}$.
\end{lemma}
\begin{proof}[Proof of Lemma~\ref{lemmainterior}]
Let $h$ be a real number  whose absolute value is sufficiently small, we define the difference quotient operator
$$D^{h}_k w(x)=\frac{w(x+h e_k)-w(x)}{h},$$
where $e_k$ is the $k$th unit vector in $R^d$.
According to Theorem 3 in Chapter 5.8 of \cite{evans1998partial}, if there exist a positive constant $\kappa$ such that $\|D^{h}_k w\|_{L^2(U)}\leq \kappa$ for all  $h$, then $\frac{\partial w}{\partial x_k}\in L^2(U)$ and $\|\frac{\partial w}{\partial x_k}\|_{L^2(U)}\leq \kappa$. We use this theorem and seek for an upper bound of
\begin{equation}\label{DhDL2}
\int_V |U^{h}_k \nabla w|^2 dx,
\end{equation}
for $k=1,...,d$ for the rest of the proof.

We derive a bound of \eqref{DhDL2} by plugging an appropriate $v$ in \eqref{weakgeneralPDE}.
Let $W$ be an open set such that $V\subsetneqq W\subsetneqq U$. We select a smooth function $\zeta$ such that
\begin{equation*}
\zeta=1 \mbox{ on } V, \qquad \zeta=0 \mbox{ on } W^c,\qquad \mbox{ and } 0\leq\zeta\leq 1.
\end{equation*}
We plug
\begin{equation*}
v=-D^{-h}_k(\zeta^2 D^{h}_k w)
\end{equation*}
into \eqref{weakgeneralPDE}, and have
\begin{equation}\label{eqAB}
-\int_D \nabla w^T A\nabla[ D^{-h}_k(\zeta^2 D^{h}_k w)] dx=- \int_D f D^{-h}_k(\zeta^2 D^{h}_k w)dx.
\end{equation}
We give a lower bound of the left-hand side of \eqref{eqAB} and an upper bound of the right-hand.
We use two basic formulas that are similar to integration by part and derivative of product respectively. For any functions $w_1,w_2\in L^2(U)$, such that $w_2(x)=0$ if  $dist(x,\partial U)<h$, we have
$$
\int_Dw_1D_k^{-h}w_2dx=-\int_D D^{h}_kw_1  w_2dx \mbox{ and } D^{h}_k(w_1w_2)=w_1^{h}D^{h}_k w_2+w_2D^{h}_k w_1,
$$
where we define $w_1^{h}(x)=w_1(x+h e_k)$. Similarly, we define the matrix function $A^h=A(x+h e_k)$.
Applying the above formulas to the left hand side of \eqref{eqAB}, we have
\begin{eqnarray*}
&&-\int_D \nabla w^T A\nabla[ D^{-h}_k(\zeta D^{h}_k w)] dx\\
&=&\int_D D_k^{h}(\nabla w^T A)\nabla(\zeta^2 D^{h}_k w)dx\\
&=&\int_D D_k^{h}(\nabla w^T) A^{h}\nabla(\zeta^2 D^{h}_k w)+
\nabla w^TD_k^{h} A\nabla(\zeta^2 D^{h}_k w)dx
\\
&=&\underbrace{\int_D \zeta^2 D^{h}_k \nabla w^T A^{h} D^{h}_k \nabla w dx}_{\mbox{$J_1$}}
+\underbrace{\int_D 2\zeta(D^{h}_k \nabla w^T A^{h}\nabla\zeta) D^{h}_k w+ 2\zeta(\nabla w^TD_k^{h} A\nabla\zeta)D^{h}_k w+\zeta^2 \nabla w^TD_k^{h} A D^{h}_k \nabla wdx}_{\mbox{$J_2$}}.
\end{eqnarray*}
 $J_1$ in the above expression has a lower bound
\begin{equation*}
J_1\geq A_{\min}\int_D \zeta^2  |D^{h}_k \nabla w|^2 dx
\end{equation*}
due to the positively definitiveness of $A(x)$. 
$|J_2|$ is bounded above by
\begin{equation}\label{A2bound1}
|J_2|\leq \kappa\|A\|_{C^1(\bar{U})}\Big( \int_D\zeta |D_k^{h}\nabla w||D^{h}_k w|+\zeta |\nabla w||D^{h}_k w|+\zeta |\nabla w||D^{h}_k\nabla w|dx\Big).
\end{equation}
The expression \eqref{A2bound1}
 can be further bounded by
\begin{equation}\label{boundA2}
|J_2|\leq\frac{A_{\min}}{2}\int_D \zeta^2 |D^{h}_k\nabla w|^2 dx+\kappa \|A\|_{C^1(\bar{U})}\times(1+\frac{\|A\|_{C^1(\bar{U})}}{A_{\min}})\int_W |\nabla w|^2+|D^{h}_k w|^2 dx.
\end{equation}
thanks to Cauchy-Schwarz inequality.
According to Theorem 3 in Chapter 5.8 of \cite{evans1998partial}, 
\begin{equation}\label{quotientbound}
\int_W|D^{h}_k w|^2 dx\leq \kappa\int_W |\nabla w|^2 dx.
\end{equation}
Therefore, \eqref{boundA2} is bounded above by
\begin{equation}\label{A2bound2}
|J_2|\leq \frac{A_{\min}}{2}\int_D \zeta^2 |D^{h}_k\nabla w|^2 dx+\kappa^2 \|A\|_{C^1(\bar{U})}\times(1+\frac{\|A\|_{C^1(\bar{U})}}{A_{\min}})\int_W |\nabla w|^2 dx.
\end{equation}
Combining \eqref{A2bound1} and \eqref{A2bound2}, we have
\begin{equation}\label{LHSlower}
\mbox{LHS of }\eqref{eqAB}=J_1+J_2\geq J_1-|J_2|\geq  \frac{A_{\min}}{2}\int_D \zeta^2 |D^{h}_k\nabla w|^2 dx-\kappa^2 \|A\|_{C^1(\bar{U})}\times(1+\frac{\|A\|_{C^1(\bar{U})}}{A_{\min}})\int_W |\nabla w|^2dx.
\end{equation}
We proceed to an upper bound of the right hand side of \eqref{eqAB}.
According to \eqref{quotientbound}, we have
\begin{eqnarray}
&&\int_D |D^{-h}_k(\zeta^2 D^{h}_k w)|^2dx\notag\\
&\leq& \kappa \int_D |\nabla(\zeta^2 D^{h}_k w)|^2 dx\notag\\
&\leq& \kappa \int_W 4|D^{h}_k w|^2|\nabla\zeta|^2\zeta^2+\zeta^2 |D^{h}_k \nabla w|^2 dx\notag\\
&\leq& \kappa^3\int_W |\nabla w|^2 +\zeta^2|D^{h}_k \nabla w|^2dx.\label{Bbound2}
\end{eqnarray}
Apply Cauchy's inequality to the right-hand side of \eqref{eqAB}, we have
\begin{equation}\label{Bbound1}
\mbox{RHS of \eqref{eqAB}}\leq \int_D |f||D^{-h}_k(\zeta^2 D^{h}_k w)|dx\leq \frac{2\kappa^3}{A_{\min}}\int_D |f|^2dx+\frac{A_{\min}}{4\kappa^3}\int_D |D^{-h}_k(\zeta^2 D^{h}_k w)|^2dx.
\end{equation}
We combine \eqref{Bbound2} and \eqref{Bbound1},
\begin{equation}\label{Bbound3}
\mbox{RHS of \eqref{eqAB}}\leq \frac{A_{\min}}{4}\int_W \zeta^2|D^{h}_k \nabla w|^2 dx 
+\frac{A_{\min}}{4}\int_W |\nabla w|^2 dx+\frac{2\kappa^3}{A_{\min}}\int_W |f|^2 dx.
\end{equation}
Combining \eqref{LHSlower} and \eqref{Bbound3}, we have
\begin{equation}\label{H2toH1bound}
\int_D \zeta^2|D^{h}_k \nabla w|^2 dx\leq  \frac{8\kappa^3}{A_{\min}^2}\int_W |f|^2 dx
+\Big[1+4\kappa^2\|A\|_{C^1(\bar{U})}\frac{\|A\|_{C^1(\bar{U})}+A_{\min}}{A_{\min}^2}\Big]\int_W|\nabla w|^2 dx.
\end{equation}
Therefore,
\begin{equation}\label{H2toH1bound2}
\int_D \zeta^2|D^{h}_k \nabla w|^2 dx\leq \kappa \frac{\max(\|A\|_{C^1(\bar{U})},1)^2}{\min(A_{\min},1)^2}\Big(\int_W|f|^2 dx + \int_W |\nabla w|^2\Big).
\end{equation}
Now we give an upper bound of $\int_D |\nabla w|$ by taking $v={\tilde \zeta}^2 w$ in \eqref{weakgeneralPDE}, where we choose ${\tilde \zeta}$ to be a smooth function such that ${\tilde \zeta}=1$ on $W$ and ${\tilde \zeta}=0$ on $U^c$. Using similar arguments as that for \eqref{H2toH1bound2}, we have
\begin{equation}\label{H1bound}
\int_W|\nabla w|^2 dx\leq \kappa \frac{\max(\|A\|_{C^1(\bar{U})},1)^2}{\min(A_{\min},1)^2}\Big(\int_W|f|^2 dx + \int_W |\nabla w|^2\Big).
\end{equation}
\eqref{H2toH1bound2} and \eqref{H1bound} together give
\begin{equation}\label{boundk}
\int_D\zeta^2|D^{h}_k\nabla w|^2 dx\leq \kappa \frac{\max(\|A\|_{C^1(\bar{U})},1)^4}{\min(A_{\min},1)^4} \int_D |f|^2 +| w|^2 dx.
\end{equation}
We complete our proof by combining \eqref{boundk} for all $k=1,...,d.$
\end{proof}

\begin{proof}[Proof of Lemma~\ref{h2boundary}]
We first consider a  special case when $U$ is a half ball
$$
U=B^{0}(0,1)\cap {R}^d_{+}.
$$
Let $V= B^{0}(0,\frac{1}{2})\cap {R}^d_+$, and select a smooth function $\zeta$ such that
$$
\zeta=1 \mbox{ on } B(0,\frac{1}{2}),  \zeta=0 \mbox{ on } B(0,1)^c, \mbox{ and } 0\leq\zeta\leq1.
$$
For $k=1,...,d-1$, we plug 
$$
v=-D^{-h}_k(\zeta^2D^{h}_k w)
$$
into \eqref{weakgeneralPDE}.
Using the same arguments for deriving \eqref{H2toH1bound} as in the proof for Lemma~\ref{lemmainterior}, we obtain that
\begin{equation*} %\label{vbound1}
\int_V |D^{h}_{k}\nabla w|^2 dx\leq \kappa  \frac{\max(\|A\|_{C^1(\bar{U})},1)^2}{\min(A_{\min},1)^2} \int_W |f|^2 +|\nabla w|^2 dx.
\end{equation*}
The above display holds for arbitrary $h$, so we have
\begin{equation}\label{h2bound1}
\sum_{i,j=1,i+j<2d}^d\int_V |\frac{\partial^2 w}{\partial x_i\partial x_j}|^2 dx\leq\kappa \frac{\max(\|A\|_{C^1(\bar{U})},1)^2}{\min(A_{\min},1)^2} \int_W |f|^2 +|\nabla w|^2 dx.
\end{equation}
We proceed to an upper bound for 
$$
\int_V |\frac{\partial^2 w}{\partial x_d \partial x_d }|^2dx.
$$
%We also have an upper bound for $|w|_{H^1_0(V)}$ by taking $v=w$ in \eqref{weakgeneralPDE},
%\begin{equation}\label{h1bound1}
%\int_D|\nabla w|^2 dx\lesssim\frac{1}{A_{\min}} \int_D(f^2+w^2)dx.
%\end{equation}
According to the remark (ii) after Theorem 1 in Chapter 6.3 of \cite{evans1998partial}, with the interior regularity obtained by Lemma~\ref{lemmainterior}, $w$ solves  \eqref{generalPDE2} almost everywhere in $U$. Consequently,
\begin{equation*}
A_{dd}\frac{\partial^2 w}{\partial x_d\partial x_d}=-\sum_{i,j=1,i+j<2d}^d{A_{ij}\frac{\partial^2 w}{\partial x_i\partial x_j}}-\sum_{i,j=1}^d  \frac{\partial A_{ij}}{\partial x_j}\frac{\partial w}{\partial x_i}-f\mbox{ a.e. }
\end{equation*}
Note that $A_{dd}\geq A_{\min}$, so the above display implies that
\begin{equation*}
|\frac{\partial^2 w}{\partial x_d\partial x_d}|\leq \kappa \frac{\|A\|_{C^1(\bar U)}}{A_{\min}}\Big(\sum_{i,j=1,i+j<2d}^d |\frac{\partial^2 w}{\partial x_i\partial x_j}|+|\nabla w|+|f|\Big).
\end{equation*}
Combining the above display with \eqref{h2bound1}, we have
\begin{equation*}
\|w\|_{H^2(V)}\leq \kappa \frac{\max(\|A\|_{C^1(\bar{U})},1)^{2}}{\min(A_{\min},1)^{2}}\Big(\||\nabla w|\|_{L^2(U)}+\|f\|_{L^2(U)}\Big).
\end{equation*}
According to \eqref{H1bound}, the above display implies
\begin{equation*} %\label{boundarybound1}
\|w\|_{H^2(V)}\leq \frac{\max(\|A\|_{C^1(\bar{U})},1)^{4}}{\min(A_{\min},1)^{4}}\Big(\| w\|_{L^2(U)}+\|f\|_{L^2(U)}\Big).
\end{equation*}
Similar to the proof for Proposition~\ref{higherorderregularity}, this result can be extended  to the case where $U$ has a twice differentiable boundary. We omit the details.
\end{proof}

\begin{proof}[Proof of Lemma~\ref{lemmainteriorhigher}]
We use induction to prove  Lemma~\ref{lemmainteriorhigher}. When $k=1$, Lemma~\ref{lemmainterior} gives
$$
\|w\|_{H^2(V)}\leq \kappa_i(A,1)\Big(\|f\|_{L^2(U)}+\|w\|_{L^2(U)}\Big).
$$
Suppose for $k=1,...,m$, Lemma~\ref{lemmainteriorhigher} holds. We intend to prove that for $k=m+1$,
\begin{equation*} %\label{eqwant}
\|w\|_{H^{m+2}}(V)\leq\kappa_i(A,m+1)\Big(\|f\|_{H^{m}(U)}+\|w\|_{L^2(U)}\Big).
\end{equation*}
By induction assumption, we have $w\in H^{m+1}_{loc}(U)$ and for any $W$ such that $V\subsetneq W\subsetneq U$
\begin{equation}\label{eqinductionassump}
\|w\|_{H^{m+1}(W)}\leq \kappa_i(A,m)\Big(\|f\|_{H^{m-1}(U)}+\|w\|_{L^2(U)}\Big).
\end{equation}
Denote by $\alpha=(\alpha_1,..,\alpha_d)^T$ a multiple index with $|\alpha|=\alpha_1+...+\alpha_d=m$.
With similar arguments as for 
\eqref{eqnewPDE}, we have that ${\bar w}=D^{\alpha} w$ is a weak solution to the PDE \eqref{eqnewPDE} without boundary condition.
Similar to the derivation for \eqref{eqsmallalpha}, $w\in H^{m+2}(V)$ and
\begin{equation*} %\label{eqbound}
\|w\|_{H^{m+2}(V)}\leq \kappa_i(A,1)\kappa_i(A,m)\max(\|A\|_{C^{m+1}(\bar U)},1)\Big(\|f\|_{H^{m}(U)}+\|w\|_{L^2(U)}\Big).
\end{equation*}
We complete the proof by induction.
\end{proof}

\section{Triangularization}\label{sectri}
The triangularization $\mathcal{T}_n$ is a partition of $U$ into triangles parametrized with the mesh size $\max_{K\in \mathcal{T}_h}\mbox{diam}(K)=O(2^{-n})$, and satisfies the following properties, 
\begin{itemize}
\item[(1)] $\bar{U}\subset\cup_{K\in \mathcal{T}_n}K$;
\item[(2)] For any $ K\in\mathcal{T}_n$, the vertices of $K$ lie either all in $\bar U$ or all in $U^c$;
\item[(3)] For $K,K'\in\mathcal{T}_n$, $K\neq K'$, $int(K)\cap int(K')=\emptyset$， where $int(K)$ denote the interior of the triangle $K$;
\item[(4)] If $K\neq K'$ but $K\cap K'\neq\emptyset$, then $K\cap K'$ is either a point or a common edge of $K$ and $K'$.
\end{itemize}

\begin{example}\label{extr}
Here we provide an example of $V_n$ and $\mathcal{T}_n$ defined over the region $U=(0,1)^2$. The detailed definition of  $\mathcal{T}_{n}$ and the finite dimensional subspace $V_n$ is given in Appendix~\ref{sectri}. In Figure~\ref{fig:tri}, 
$\mathcal{T}_n$ is the set of triangles that partitions $(0,1)^2$. The shaded area is the support for the basis function $\phi_1$ of the space $V_2$. In particular, $\phi_1$ is a piecewise linear function on each triangle (and is constant if the triangle is outside the support) and $\phi_1(0.25,0.25)=1$, $\phi_1(0.25,0)=\phi_1(0.5,0)=\phi_1(0.5,0.25)=\phi_1(0.25,0.5)=\phi_1(0,0.5)=0.$ Similar basis functions $\phi_2,...,\phi_9$ can be constructed corresponding to the nine inner nodes (circled points in Figure~\ref{fig:tri}).

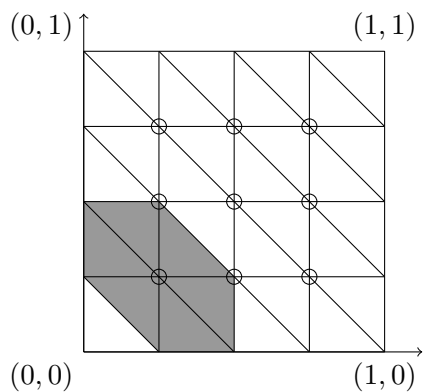
\begin{figure}[!ht]
\centering
\begin{tikzpicture}

  \foreach \x in {0,...,4}
    \foreach \y [count=\yi] in {0,...,3}  
	  {
	  
      \draw (\x,\y)--(\x,\yi)  (\y,\x)--(\yi,\x) ;
      }
   \foreach \x in {0,...,4}
      \foreach \y [count=\yi] in {0,...,4}{
      {
      	  \pgfmathtruncatemacro{\xylabel}{\x + \y*5+1}
%            \draw (\x,\y) node[below left]{ $x_{\xylabel}$}  ;
            }
      }
      \draw(0,0) node[below left]{$(0,0)$};
      \draw(4,0) node[below ]{$(1,0)$};
      \draw(0,4) node[above left]{$(0,1)$};
      \draw(4,4) node[above ]{$(1,1)$};
  \foreach \x [count=\xi]in {0,...,3}
    \foreach \y [count=\yi] in {0,...,3}  
	  {
		\draw (\x,\yi) -- (\xi,\y);
      }
  \foreach \x [count=\xi]in {1,...,3}
    \foreach \y [count=\yi] in {1,...,3}  
	  {
		\draw (\x,\y) circle(0.1cm);
      }
      
   \draw[->] (-0,0) --(4.5,0);
   \draw[->] (-0,0) --(0,4.5);
	\fill [opacity=0.4]
	{(1,0) -- (2,0) -- (2,1) -- (1,2) -- (0,2) -- (0,1) };
\end{tikzpicture}
\caption{Triangularization $\mathcal{T}_2$ on $(0,1)^2.$ \label{fig:tri}}

\end{figure}
\end{example}

\end{document}